\documentclass{article}
\usepackage[letterpaper]{geometry}
\usepackage{graphicx} 
\usepackage{amsmath}
\usepackage{amsfonts} 
\usepackage{amsthm} 
\usepackage{amssymb}

\usepackage{authblk}

\usepackage{subcaption}
\usepackage{dsfont}
\usepackage{hyperref}
\usepackage{xcolor}
\usepackage{bbm}
\usepackage{comment}
\usepackage[normalem]{ulem}
\newtheorem{theorem}{Theorem}[section]
\newtheorem{proposition}{Proposition}[section]
\newtheorem{corollary}{Corollary}[section]
\newtheorem{lemma}{Lemma}[section]
\newtheorem{remark}{Remark}[section]

\newcommand{\Out}{\mu_{\nsim}}

\newcommand{\Ex}{\mathbb{E}}

\newcommand{\LRvertiii}[1]{{\left\vert\kern-0.25ex\left\vert\kern-0.25ex\left\vert #1 
    \right\vert\kern-0.25ex\right\vert\kern-0.25ex\right\vert}}

\newcommand{\Bigvertiii}[1]{\Big\vert\kern-0.25ex\Big\vert\kern-0.25ex\Big\vert #1     \Big\vert\kern-0.25ex\Big\vert\kern-0.25ex\Big\vert}

\newcommand{\vertiii}[1]{\vert\kern-0.25ex\vert\kern-0.25ex\vert #1     \vert\kern-0.25ex\vert\kern-0.25ex\vert}

\title{
First and second order firing statistics 
of large neuronal networks with community structure}

\author{
Thiago H. N. Coelho\textsuperscript{1}, 
Daniel Figueiredo\textsuperscript{1}, 
Giulio Iacobelli\textsuperscript{2}
\\
Guilherme Ost\textsuperscript{3}, 
Conrado C. Pinto\textsuperscript{1}
}

\date{}

\begin{document}

\maketitle

\footnotetext[1]{Computer Sciences and Systems Engineering Department, Universidade Federal do Rio de Janeiro (UFRJ), Brazil.}
\footnotetext[2]{Mathematical Institute, Federal University of Rio de Janeiro (UFRJ), Brazil.}
\footnotetext[3]{IMPA, Brazil.}

\begin{abstract}
We study a discrete-time stochastic model for the spiking activity of a network of $n$ interacting neurons partitioned into $k$ communities. 
The  interaction network is defined on a random directed weighted graph generated by a Stochastic Block Model with $k+1$ edge types ($k$ intra-community, $1$ inter-community). 
Each edge is present independently with a probability that depends only on its type. 
When present, an edge carries a type-dependent deterministic weight scaled by $n^{-1}$, which is excitatory or inhibitory depending on  an  independent random sign.

Conditionally on the  interaction network, the spiking process evolves as a stationary Markov chain. 
Our results describe how the network's community structure manifests itself in 
the first- and second-order statistics of this chain, as $n \to \infty$ with $k$ fixed.
A key step in our analysis is to represent these statistics as explicit functions of the random interaction network. Using these representations, we prove that, with high probability:
 (i) the vectors of mean firing activity and firing  variances are approximated in the supremum norm by deterministic vectors that are constant within each community, with an error of  order $\sqrt{\log(n)/n}$;  (ii) the off-diagonal part of the simultaneous ($0$-lagged) firing covariance matrix is approximated in max norm,  with an error of order $\sqrt{\log(n)/n^3}$, by a deterministic community-block matrix (constant on each block defined by a pair of communities), with the block values determined by the unique solution to a $k \times k$ Stein matrix equation;
(iii) in contrast, the $t$-lagged covariance matrix, for every lag $t \ge 1$, is approximated in max norm, again with an error of order $\sqrt{\log(n)/n^3}$, by a random  matrix
built from two components: a random matrix capturing the interaction network and a deterministic community-block matrix. 
Finally, we discuss conditions on the model parameters under which these limiting statistics enable the identifiability of the underlying community structure.
\end{abstract}

\section{Introduction}
\subsection{Model and problem statement}
\label{sec:model:descr:prob:state}
In this paper, we consider a stochastic, discrete-time model describing the spiking activity of a network of $n$ interacting neurons partitioned into $k$ communities.
The model has two components, which can be informally described as follows. 

The first is a directed weighted graph with $n$ vertices, whose edges can be of $k+1$ types: $k$ {\em intra-community} types, one for each community, and a single {\em inter-community} type. An edge is of intra-community type if it connects two neurons in the same community, and of inter-community type otherwise. 
Each edge in the graph carries a weight whose absolute value is the same for all edges of a given type; we refer to this common value as the {\em intensity} of the edge type.
The sign of the weight, however, 
is independent of the edge type: edges with positive weights are called {\em excitatory}, while those with negative weights are called {\em inhibitory}.
For most of our results, the graph structure and the edge signs are random. Specifically, edges are present independently of one another, with probabilities depending only on their types, as in a {\em Stochastic Block Model} \cite{JMLR:v18:16-480}. Independently of the edge configuration, the edge signs are independent and identically distributed, with each sign being positive with a fixed probability. 
We write $\Theta$ for the matrix indicating which edges are present in the graph and $H$ for the matrix of signs. Hereafter, we call the pair
$(\Theta, H)$ the {\em environment} and denote by $A=(A_{ij})_{i,j\in [n]}$ its associated {\em interaction matrix}. The entry $A_{ij}$ represents  the weight of the edge from neuron $i$ to neuron $j$, scaled by $n^{-1}$, if  the edge is present, and is $0$ otherwise.

The second component is a sequence $X=(X_t)_{t\in\mathbb{N}}$ of random vectors taking values in $\{0,1\}^n$, representing the time evolution of the spikes of the $n$ neurons. Writing $X_{i,t}$ to denote the $i$-th coordinate of the random vector $X_t$, we have $X_{i,t}=1$ if  neuron $i$ spikes at time $t$, and $0$ otherwise. 
Conditionally on the realization of the environment $(\Theta, H)$, the sequence $X$ forms a Markov chain. Specifically,
given the environment and the entire configuration of spikes at time $t-1$, each neuron $i$ spikes independently of the others at the next time step, with probability
\begin{equation}
\label{def:trans:prob:informal}
 \mathbb{P}_{(\Theta,H)}\bigl(X_{i,t}=1 \mid X_{t-1}\bigr)
  = \lambda + \big(A^{\top}X_{t-1}\big)_i ,    
\end{equation}
where $A^{\top}$ denotes the transpose of the matrix $A$ and $\lambda \in [0,1]$ represents the {\em spontaneous firing rate}, which is the same for all neurons. The second term on the right-hand side of \eqref{def:trans:prob:informal} is a weighted sum of the spikes emitted at time $t-1$ by
the neurons that point to neuron $i$ in the graph, where the weights are given by the corresponding entries of the interaction matrix $A$. Since the entries of $A$ are of order $n^{-1}$, the interactions are individually weak.

Throughout our analysis, we work under conditions on the model parameters (see \eqref{def:M} and \eqref{eq:wellposed_random_scenario}) ensuring that the right-hand side of \eqref{def:trans:prob:informal} is always between $0$ and $1$ for a set of realizations of the environment $(\Theta,H)$ whose probability approaches $1$ as $n\to\infty$ (this is formalized in Lemma \ref{lem:random_NFC_welldefined}). We refer to such realizations as {\em admissible} realizations.
Conditionally on admissible realizations of the environment, the Markov chain $X$ admits a unique invariant measure, and we always work with its stationary version. In the sequel, we denote by $\mathbb{P}_{(\Theta,H)}$ the law of the stationary chain $X$ given the environment $(\Theta,H)$ and by $\mathbb{E}_{(\Theta,H)}$ the corresponding expectation.

In what follows, let $\mathbb{E}_{(\Theta,H)}[X_t]$ be the column vector of the mean firing activity at time $t$
and let $\Sigma^{(t)}_{(\Theta,H)}$ denote the $t$-lagged firing covariance matrix that is defined as
\begin{equation*}
\Sigma^{(t)}_{(\Theta,H)}:=\Ex_{(\Theta,H)}[X_t(X_0)^{\top}]-\Ex_{(\Theta,H)}[X_t](\Ex_{(\Theta,H)}[X_0])^{\top}, \ t\in\mathbb{N} \,.
\end{equation*}
By stationarity, the mean vectors $\mathbb{E}_{(\Theta,H)}[X_t]$ do not depend on $t$, and in the rest of the paper we refer to $\mathbb{E}_{(\Theta,H)}[X_0]$ as the vector of mean firing rates. Note that both this vector and the lagged covariance matrices are random, depending on the realization of the environment. 
Also, recall that the law of this environment is determined by the partition of the population into the $k$ communities, the edge probabilities, the intensities of the different edge types, and the probability that an edge is excitatory.

This leads to the following related problems, which we address in this paper.      

{\it Problem statement}. As the size $n$ of the neuronal network
grows while the number $k$ of communities is kept fixed, how does the community structure of the environment manifest itself in the mean firing rates
and in the lagged firing covariance matrices? Conversely, to what extent can the partition of the population into communities be recovered from these quantities alone?

\subsection{
Overview of the main results
}

Throughout the paper, we denote by $\mathcal{C}=\{C_1, \ldots, C_k\}$ the partition of the set $[n]:=\{1,\ldots,n\}$ representing the $k$ communities of the entire network of neurons, whose size is $n$. In our analysis, we assume that all communities have sizes proportional to the total number of neurons in the network, i.e., we suppose that $|C_b|=\alpha_b n\geq 1$ for every $b \in [k]$, for some positive values $\alpha_1,\ldots, \alpha_k$ satisfying $\alpha_1 + \ldots + \alpha_k=1$. In what follows, we write  $\|\cdot\|_{\infty}$ to denote the usual supremum norm in $\mathbb{R}^n$ and $K$ to indicate a positive constant independent of $n$.
 
Recall that $\mathbb{E}_{(\Theta,H)}[X_0]$ denotes the random vector of mean firing rates, where the randomness comes from the random environment $(\Theta,H)$. In our first main result, Theorem \ref{thm:convergence_mean_version2}, we show how the community structure of the environment is encoded in this random vector when $n$ is large. 
Specifically, we show that with high probability
\[
\| \mathbb{E}_{(\Theta,H)}[X_0]-\overline{m}\|_{\infty}\leq K \sqrt{\frac{\log(n)}{n}},
\]    
where $\overline{m}$ is the $n$-dimensional vector defined as
\begin{equation}
\label{def:mean:field:approximation:for:the:mean}
\overline{m}:=\lambda \sum_{b \in [k]}\zeta_b \mathbf{1}_{C_b},
\end{equation}
where $\zeta=(\zeta_1,\ldots,\zeta_k)$ is an explicitly known $k$-dimensional vector (see,  \eqref{def:limit:vector:firing:probabilities}) and $\mathbf{1}_{C_b}$ denotes the $n$-dimensional column vector with entries in $C_b$ equal to $1$ and $0$ otherwise. 
Note that all entries of $\overline{m}$ corresponding to neurons in community $C_b$ share the same value $\lambda\zeta_b$, that is, $\overline{m}$ is constant within each community.  The vector $\zeta$ can be explicitly computed from the $k\times k$ matrix $W$ defined in \eqref{def:matrix_G} whose entry $W_{a b}$ corresponds to the mean cumulative effect that community $C_b$ exerts on  community $C_a$.

The proof of Theorem \ref{thm:convergence_mean_version2} can be divided into three steps. In the first step, we prove that for all admissible realizations of the environment, we can write 
\[
\Ex_{(\Theta,H)}[X_0] = (I_n-A^{\top})^{-1}\lambda\mathbf{1}_{n}, \]
where $I_n$ is the $n$-dimensional identify matrix and $\mathbf{1}_{n}$ is the $n$-dimensional column vector with all entries equal to $1$. This is a consequence of Proposition \ref{prop:mean} stated in Section \ref{sec:preliminary:results}.

The second step consists in quantifying the error of approximating the random vector $\Ex_{(\Theta,H)}[X_0]$ by the deterministic vector $(I_n-\Ex^{(n)}[A^{\top}])^{-1}\lambda\mathbf{1}_{n}$, where $\Ex^{(n)}$ denotes the expectation taken with respect to the randomness of the environment $(\Theta,H)$. This follows from Lemma \ref{lem:1} of Appendix \ref{sec:appendix:asymp:app:mean:variance}
and Hoeffding's inequality.
The final step is then to show that $(I_n-\Ex^{(n)}[A^{\top}])^{-1}\lambda\mathbf{1}_{n}=\overline{m}$ inherits the block structure of the environment $(\Theta,H)$. 
Lemma \ref{lem:EAt:power:1Ca} of Appendix \ref{sec:mean:A:and:L} establishes the main result needed in the proof of this final step. 

As a simple consequence of Theorem \ref{thm:convergence_mean_version2}, we 
show that the vector of firing variances, defined as
\[
v_{(\Theta,H)}(X_0):=(\mathbb{E}_{(\Theta,H)}[X_{1,0}](1-\mathbb{E}_{(\Theta,H)}[X_{1,0}]),\ldots, \mathbb{E}_{(\Theta,H)}[X_{n,0}](1-\mathbb{E}_{(\Theta,H)}[X_{n,0}]))^{\top}, 
\]
satisfies with high probability
\[
\big\| v_{(\Theta,H)}(X_0)-\overline{v}\big\|_{\infty}\leq K\sqrt{\frac{\log(n)}{n}}\,,
\]
where the $n$-dimensional vector 
\begin{equation}
\label{def:mean:field:approx:for:variance}   
\overline{v}:=(\overline{m}_1(1-\overline{m}_1),\ldots, \overline{m}_n(1-\overline{m}_n))^{\top}
\end{equation}
is also constant
within each community (see Corollary \ref{cor:convergence_variance}).

Next, we describe how the community structure of the environment emerges in the lagged firing covariances $\Sigma^{(t)}_{(\Theta,H)}$ as $n\to\infty$. We start describing what happens for the simultaneous firing covariance, that is for $t=0$. In this case, since the diagonal terms of $\Sigma^{(0)}_{(\Theta,H)}$ are already known (they are given by the vector of firing variances $v_{(\Theta,H)}(X_0)$), we only need to focus on the behavior of its off-diagonal entries. 
The asymptotic behavior of these entries is described by means of a $k\times k$ matrix $\Psi$ that is the unique solution of the Stein matrix equation  
\begin{equation*}
\Psi-W\Psi W^{\top}=WDW^{\top}, 
\end{equation*}
where $D$ is a diagonal matrix whose non-null entries are expressed in terms of the vector $\zeta$ defined in \eqref{def:limit:vector:firing:probabilities} 
and the parameters $\alpha_b$ (see \eqref{def:stein:equation}). Specifically, introducing the $n\times n$ matrix 
\begin{equation}
\label{def:mean:field:approximation:for:the:simul:cov}
\overline{\Sigma}^{(0)}:=\frac{1}{n}\sum_{a,b\in[k]}\Psi_{ab}\,\mathbf{1}_{C_a}\mathbf{1}_{C_b}^{\top},
\end{equation}
in Theorem \ref{thm:convergence_cov} we establish that with high probability   
\[
\max_{i\neq j} \Big|\Sigma_{(\Theta,H),ij}^{(0)}-\overline{\Sigma}^{(0)}_{ij}\Big|
\leq K\sqrt{\frac{\log(n)}{n^3}}.
\]
We see immediately from \eqref{def:mean:field:approximation:for:the:simul:cov} that the matrix $\overline{\Sigma}^{(0)}$ has a block structure: the entries $\overline{\Sigma}^{(0)}_{ij}$ are equal to $n^{-1}\Psi_{ab}$ for all pairs $(i,j)\in C_a\times C_b$. In particular, this implies that asymptotically as $n\to\infty$ the off-diagonal entries of simultaneous firing covariance are of order $n^{-1}$, whereas its diagonal entries are of
order $1$ 
(see Corollary \ref{cor:convergence_variance}).

At the high level, the proof of Theorem \ref{thm:convergence_cov} is similar to that of Theorem \ref{thm:convergence_mean_version2}, except that we do not work directly with the simultaneous firing covariance $\Sigma^{(0)}_{(\Theta,H)}$. Instead, we work with its vectorized form. This means that we work with the $n^2$-dimensional column vector obtained by stacking the columns of $\Sigma^{(0)}_{(\Theta,H)}$ one by one, denoted by ${\rm vec}\big( \Sigma^{(0)}_{(\Theta,H)} \big)$. Using this representation, to prove Theorem \ref{thm:convergence_cov}, we first show, by relying on Proposition \ref{prop:cov:vec:form}, that for all admissible realizations of the environment 
\[
{\rm vec}\big( \Sigma^{(0)}_{(\Theta,H)} \big)=(I_{n^2}-L)^{-1}{\rm vec}\big({\rm Diag}(v_{(\Theta,H)}(X_0))),
\]
where ${\rm vec}\big({\rm Diag}(v_{(\Theta,H)}(X_0)))$ denotes the vectorized diagonal matrix whose non-null entries are given by the vector of firing variances $v_{(\Theta,H)}(X_0)$ and $L$ is a $n^{2} \times n^{2}$ matrix whose entries are defined as: 
\begin{equation*}
    L_{u v} = \begin{cases}
       0 & \text{ if $u=n(i-1)+i$, for $i \in [n]$ }
       \\
        (A^{\top} \otimes  A^{\top})_{u v} & \text{ otherwise}
    \end{cases}
    \,.
\end{equation*}
To conclude, we estimate the error of approximating the random vector ${\rm vec}\big( \Sigma^{(0)}_{(\Theta,H)} \big)$ by the deterministic vector $(I_{n^2}-\Ex^{(n)}[L])^{-1}{\rm vec}\big({\rm Diag}(\overline{v})\big)$ and then show in Proposition \ref{prop:gen:Psi} stated in Appendix \ref{sec:asym:approx:syncronous:cov:matrix} that this latter vector is, in turn, sufficiently close to the vectorized version of the matrix $\overline{\Sigma}^{(0)}$.  

Unlike the previous approximations, the behavior of the $t$-lagged firing covariance matrices for every
$t\geq 1$ is described by a random matrix denoted by $\overline{\Sigma}^{(t)}_{(\Theta,H)}$, which is defined as
\begin{equation}
\label{def:gen:Gamma:t}
\overline{\Sigma}^{(t)}_{(\Theta,H)}:=\big(A^{\top}_{(\Theta,H)}\big)^{t}\,{\rm diag}(\overline{v})
+\frac{1}{n}\sum_{a,b\in[k]}\big(W^{t}\Psi\big)_{ab}\,\mathbf{1}_{C_a}\mathbf{1}_{C_b}^{\top}\,,
\end{equation}
with $\overline{v}$ and $W$ given by \eqref{def:mean:field:approx:for:variance} and \eqref{def:matrix_G}, respectively, and $\Psi$ as above (see \eqref{def:Psi} for a rigorous definition of $\Psi$).  This is the content of Theorem \ref{thm:gen:convergence_cov_lag}, where we show that with high probability
\[
\max_{i,j} \Big|\Sigma_{(\Theta,H),ij}^{(t)}-\overline{\Sigma}^{(t)}_{ij}\Big|
\leq K\sqrt{\frac{\log(n)}{n^3}}.
\]
The proof idea of Theorem \ref{thm:gen:convergence_cov_lag} is as follows. We first use the identity
$\Sigma^{(t)}_{(\Theta,H)}=(A^{\top}_{(\Theta,H)})^{t}\Sigma^{(0)}_{(\Theta,H)}$ which holds for all admissible realization of the environment by Proposition~\ref{prop:t-laggedCov}. Then, we decompose $\Sigma^{(0)}_{(\Theta,H)}$  as the sum of its diagonal and off-diagonal terms and apply Corollary \ref{cor:convergence_variance} and Theorem \ref{thm:convergence_cov} to these terms, respectively. In the final step, we quantify the error of replacing $(A^{\top}_{(\Theta,H)})^{t}$ by its expectation in front of the second term.

The above results describe how the first and second order statistics of the firing process $X$ inherits the community structure of the environment, as the number of neurons $n\to\infty$. 
In Subsection \ref{sec:detection}, we further investigate the extent to which the community structure can be recovered from these limiting statistics by discussing conditions on the model parameters under which the communities are identifiable. A thorough investigation of community detection in the class of models considered here, including the development and analysis of detection procedures, is left for future work.
%

\subsection{Related work}

The neuron-firing model considered in this paper is a variant of the one recently studied in \cite{Chevallier:2026:Inferring,chevallier2024community}. Unlike our model, in these papers the underlying random graph model is homogeneous, being distributed as a directed Erd\H{o}s-Rényi model with $n$ vertices and edge probability $p\in [0,1]$, a special case of the model considered here. In addition, all edges carry the same intensity, whereas here the intensity depends on the type of the edge.
Similarly to our model, the weight of each edge has a sign, though the assignment rule is different. Specifically, the edge from neuron $j$ to neuron $i$ has positive sign if neuron $j$ belongs to a set ${\cal P}_+$ and negative sign otherwise. In their framework, the sets ${\cal P}_+$ and ${\cal P}_-:=({\cal P}_+)^{c}$  are interpreted as the sets of excitatory and inhibitory neurons, respectively.

In \cite{Chevallier:2026:Inferring}, the authors show that it is possible to estimate the edge probability $p$ by observing $T$ consecutive configurations of spikes of the $n$ neurons, with rate $n^{-1/2}+n^{1/2}/T+(\log(T)/T)^{1/2}$. In particular, in the asymptotic regime in which $n$ grows and $T$ scales with $n$, the proposed estimator is proved to be consistent provided that $n^{1/2}/T\to 0$ as $n\to\infty$. In the same setting, the paper \cite{chevallier2024community} addresses the question of discriminating the sets of excitatory neurons from the inhibitory ones. 
In their main result, the authors show that this can be done by using simple aggregated and spectral methods for which the proportion of misclassified neurons converges to $0$ in probability, provided that $n/T\to 0$ (up to log factors) as $n\to\infty$. Moreover, this condition is proved to be near-optimal in the minimax sense. The methods developed in both papers rely on first and second order
statistics of the spiking activity generated by a homogeneous random graph. Here we go beyond the homogeneity assumption, by characterizing these statistics when the underlying random graph model has a block structure, in both its edge probabilities and its intensities.

The problems stated in Section \ref{sec:model:descr:prob:state} concern the relation between the architecture of a
neuronal network (which characterizes the neuronal interactions) and its firing patterns. In experimental data, we often have access to the firing patterns of large networks of neurons, while their architecture is unobserved. 
Estimating these architectures from data has attracted considerable attention in recent years. 
This has been done in the context of interacting chains with variable-length memory in \cite{duarte2019estimating}, where 
the architecture of the network is described in terms of interaction neighborhoods and explicit bounds on the probabilities of under- and overestimating these neighborhoods are
provided. In a more general context of sparse space-time models, \cite{ost2020sparse}
obtains oracle inequalities for Lasso-type estimators of the model interactions, valid even when only part of the network is observed.

From a more applied perspective, a number of works approaches the same problem
through generalized linear models (GLMs) fitted to recorded spike trains. For
example, the work \cite{kobayashi2019reconstructing} applies a GLM to the
cross-correlogram of spike train pairs, together with a thresholding procedure to
determine which inferred connections are statistically reliable, and recovers
both the weighted adjacencies of the neuronal network and the excitatory or
inhibitory character of each neuron. 
To achieve optimal performance the
parameters of the model have to be carefully tuned for each individual dataset. 
This obstacle is addressed in the subsequent work \cite{endo2021convolutional}, where a convolutional neural network trained on
simulated data replaces the GLM fitting. This removes the need for fine tuning model parameters, but the interpretation of
results is less clear. As another example, let us mention \cite{linderman2016bayesian} that
proposes a Bayesian framework to uncover latent structures in neural circuits,
such as the spatial arrangement of neurons and their functional types, combining
a GLM for spike generation with graph-theoretical priors governing the
relationship between latent features and neural connectivity patterns. 

The GLM-based methods above are assessed on synthetic and empirical data rather than being supported  by theoretical guarantees. We note that the cross-correlograms on which several of them are based correspond, to some extent, to lagged covariances of fixed pair of neurons, viewed as a function of the lag. The asymptotic description of these covariances provided in Theorems \ref{thm:convergence_cov} and \ref{thm:gen:convergence_cov_lag}  may, therefore, be an important step towards a rigorous account of what such methods can recover.

In the theoretical neuroscience literature, the relation between architecture and activity has also been studied in the opposite direction, as for example in the works \cite{ostojic2009connectivity} and \cite{pernice2011structure}, where the goal is to quantify correlations between the activity of pairs of neurons for different stochastic models and network architectures. 
In \cite{ostojic2009connectivity}, this is done for a pair of integrate-and-fire neurons, the surrounding network being represented by a stationary Gaussian input.
In \cite{pernice2011structure}, neurons are modeled by Hawkes processes, a class of interacting point processes widely used in neuroscience. For this model, the authors obtain closed expressions for the covariances integrated over the time lag, under the assumption that the network is regular, i.e., all neurons have the same in- and out-degrees. They also assess through simulations the accuracy of these expressions when the network is no longer regular.  

We conclude this section by briefly mentioning that stochastic models having a latent block structure have been recently considered in \cite{Berthet2019Exact,Bunea2020Model, peixoto2019network}. 
In the works \cite{Berthet2019Exact} and \cite{Bunea2020Model}, the observations are assumed to be i.i.d. copies of a random vector. 
In the former, the observed random vector is distributed as the Ising model without external field where the coupling weights have a block structure, while in the latter the structure
is imposed on the covariance matrix of the observed vector. In \cite{peixoto2019network}, the network and its communities are reconstructed
jointly from time series by a Bayesian procedure, without rigorously discussing guarantees of recovery.
In our setting, the observations form a single trajectory of a dependent process,
whose dynamics is Markovian conditionally on the environment. Since the associated transition probabilities depend on the environment that has different communities, one expects these communities to leave a trace in the lagged covariances; our Theorems \ref{thm:convergence_cov} and \ref{thm:gen:convergence_cov_lag} make
this precise.

\subsection{Outline of the paper}

The rest of the paper is organized as follows. In Section~\ref{sec:model}, we introduce the
neuron-firing model with communities on a fixed interaction graph and fix most of the notation used throughout the paper. Section~\ref{sec:preliminary:results} collects three
auxiliary results showing that the first and second order statistics of the firing activity can be computed exactly from the interaction matrix, for every choice of the interaction graph. In Section~\ref{sec:NFC_random}, we extend our neuron-firing
model to a random environment distributed as a stochastic block model. Our main
results are stated in Section~\ref{sec:main:proba:results}. In Section~\ref{sec:detection}, we
derive from these results the consequences for the identifiability of the
communities. 
All proofs are deferred to the appendices.

\section{Neurons firing model with communities (NFC)} \label{sec:model}

In Section~\ref{sec:NFC}, we introduce our neuron-firing model with  communities where the interactions between neurons are fixed, followed by a few preliminary results on this model, presented in Section \ref{sec:preliminary:results}.


\subsection{NFC on a fixed environment}\label{sec:NFC}
Consider integers $n\geq k\geq 1$ and let $\left[ n \right]$ denote the set of neurons and let $\mathcal{C}=\{C_1, \ldots, C_k\}$ be a partition of $[n]$ representing $k$ different communities of neurons. 
Let $\theta=(\theta_{ij})_{1\leq i,j\leq n}$ be a $\{0,1\}$ matrix encoding the pairwise (directed) interaction between neurons; $\theta_{ij}=0$ means that neuron $i$ does not interact directly with neuron $j$. Note that the matrix $\theta$ need not be symmetric which allows to have, e.g.,  $\theta_{ij}=0$ and $\theta_{ji}=1$; that is the matrix $\theta$ can be interpreted as a directed graph. 
There are $k+1$ types of edges: $k$ {\em intra-community} and $1$ {\em inter-community}, depending on whether they connect neurons in a same community among the $k$ possible ones or not, respectively. Given two neurons $i, j$, we shall write $i \overset{b}{\sim} j$ to denote that $i$ and $j$ both belong to the community $C_b$ and $i \nsim j$ to denote that they belong to different communities.  As mentioned in the introduction, the intensity of the interaction may differ depending on the edge type.  Each interaction in $\theta$ can be {\em inhibitory} or {\em excitatory} and this is encoded in a matrix $\eta=(\eta_{ij})_{1\leq i,j\leq n}$, with $\{-1,1\}$ entries. Let  $G=(\mathcal{C}, \theta, \eta)$ be a directed weighted graph with $n$ vertices and $k$ communities summarizing the network structure introduced above. Henceforth, we shall refer to $G$ as the {\em interaction graph}. 

The neurons' activity is modeled via a system of $n$ interacting $\{0,1\}$-valued random variables $\{X_{i,t}, i \in [n], t \in \{0\} \cup \mathbb{Z}^+\}$, representing whether a neuron fires or not at a given time; $X_{i,t}=1$ means that neuron $i$ fires at time $t$ and $X_{i,t}=0$ that it does not fire. These random variables are coupled via the interaction graph $G=(\mathcal{C}, \theta, \eta)$ as detailed below. Denoting by  
 $X_t=(X_{1,t}, \ldots, X_{n,t})$ the
configuration of the system of $n$ neurons at time $t$, the  process $X=(X_t)_{t\geq 0}$  evolves as a discrete time  Markov chain on the state space $\{0,1\}^n$ in which the conditional distribution of $X_t$ given that $X_{t-1} = x$, $x\in \{0,1\}^n$,  is that of $n$ independent Bernoulli random variables with parameters $\phi_{G,i}(x)$, $i \in [n]$, i.e., for all  $x,y \in \{0,1\}^n$ and $t\geq 1,$
\begin{align}\label{eq:cond_ind}
\mathbb{P}_G (X_t = y \mid X_{t-1}=x)=
\prod_{i \in [n]} \left(\phi_{G,i}(x) \right)^{y_i}\left( 1- \phi_{G,i}(x) \right)^{1-y_i}\,,    
\end{align}
where, given $x=(x_1, \ldots, x_N) \in \{0,1\}^n$, 
\begin{align}\label{eq:firingProb}
\phi_{G,i}(x) &:= \lambda + \sum_{b \in [k]} \mathbbm{1}_{\{i \in C_b\}}\Big(\frac{\mu_b}{n}\sum_{j: j\overset{b}{\sim}i}  \theta_{ji} \eta_{ji} \, x_{j} + \frac{\Out}{n} \sum_{j: j \nsim i} \theta_{ji} \eta_{ji} \,  x_{j}\Big)\,.
\end{align}

Equation~\eqref{eq:firingProb} may be interpreted as the probability of neuron $i$ firing at a given time given the system's configuration at the preceding time.  
The first inner summation covers neurons in the same community of $i$  while the second covers neurons in the other communities. The parameter $\lambda \in [0,1]$ represents the spontaneous activity. The parameters $(\mu_b)_{b\in [k]}$ and  $\Out$, all positive, codify the intensity of the interaction between neurons in the same community and in different ones, respectively. A neuron $j$  interacting directly with $i$ (i.e., $\theta_{ji}=1$) may either {\em excite} ($\eta_{ji}=1$) or  {\em inhibit} ($\eta_{ji}=-1$) neuron $i$'s activity, increasing or decreasing its firing probability by an amount of $\mu_b/n$ (resp., $\Out/n$) if  $j\overset{b}{\sim} i$   (resp., $j \nsim i$).  

It will be convenient to write \eqref{eq:firingProb} in a matrix form. Thus, given the set of neurons $[n]$, the $k$ communities $\mathcal{C}$ and the interaction graph $G=(\mathcal{C}, \theta, \eta)$, we introduce the corresponding  {\em interaction matrix}  $A = \left( A_{ij} \right)_{i,j\in [n]}$ whose  entries are defined as (to avoid clutter in the notation we omit the dependence on $G$),
 
 \begin{equation}\label{eq:matrixA}
A_{ij} :=   \sum_{b\in [k]}  \frac{\mu_b}{n} \theta_{ij} \eta_{ij} \mathbbm{1}_{\{i\overset{b}{\sim} j\}} +   \frac{\Out}{n} \theta_{ij} \eta_{ij} \mathbbm{1}_{\{i\nsim j\}}\,. 
\end{equation}
Once defined the interaction matrix $A$, for every $x \in \{0,1\}^n$ and every $i \in [n]$,  the firing probability $\phi_{G,i}(x)$ can be rewritten as 
\[
\phi_{G,i}(x) = \lambda + \big( A^{\top} x \big)_i\,.  
\]

Note that, for \eqref{eq:cond_ind} to be well defined, we should have that $\phi_{i,G}(x) \in [0,1]$ for every $i \in [n]$ and every $x \in \{0,1\}^n$. This induces some restriction on the model parameters $\lambda, (\mu_b)_{b\in [k]}, \Out$ and the interaction graph $G$. A sufficient condition is:
\begin{equation}\label{eq:wellposed}
\vertiii{A^\top}_{\infty}
< \lambda\wedge (1-\lambda)\,, 
\end{equation}
where $\vertiii{A^\top}_{\infty}:=\max_i \sum_{j \in [n]}|A^{\top}_{ij}|$ and, hereafter,  $u\wedge v$ (resp.,  $u\vee v$) denotes the minimum (resp., maximum) value between the real numbers $u$ and $v$. 

Note that $\vertiii{A^\top}_{\infty}=  \frac{1}{n}\max_{b \in [k]} \mu_b\Delta_b + \Out \frac{\Delta_{\nsim}}{n}$, where $\Delta_{b}=\max_{i\in C_b} \sum_{j:j\overset{b}{\sim} i}\theta_{ji}$ is the maximal intra degree within community $C_b$ and $\Delta_{\nsim}= \max_i\sum_{j:j\nsim i}\theta_{ji}$ is the  maximal inter degree. Note also that, for \eqref{eq:wellposed} to be satisfied, it must necessarily be the case that $\vertiii{A^\top}_{\infty}{<1/2}$. 

Throughout the paper we will work with a stationary version of the chain $X$, which in particular implies that $\mathbb{P}_G(X_t=x)=\mathbb{P}_G(X_0=x)$ for all $t$ and every $x\in \{0,1\}^n$, as well as $\mathbb{E}_G[X_{i,t}X_{j,t}]=\mathbb{E}_G[X_{i,0}X_{j,0}]$ for all $t$ and $i,j \in [n]$.
In passing, let us point out that if  \eqref{eq:wellposed} holds, then $0<\phi_{i,G}(x)<1$ for every $i \in [n]$ and every $x \in \{0,1\}^n$. The latter implies that the Markov chain is irreducible and aperiodic which, in turn, assures the existence of a unique stationary measure. 

\medskip
{\it General notation.} Before concluding this section, let us introduce some general notation used in the sequel. We write $\mathbf{1}_{n}$ to indicate the column vector of length $n$ with all entries equal to $1$. More generally, 
For any subset $S\subset [n]$, we denote $\mathbf{1}_{S}$ the column vector of length $n$ with entries in $S$ equal to $1$ and $0$ otherwise.
For a $n$-dimensional vector $w$, we denote $||w||_{\infty}:=\max_{i\in [n]}|w_i|$ to indicate its supremum norm. 
 For a $n^2$-dimensional column vector $w$, we will often write $w_{ij}$ to indicate its entry $w_u$, where $u=n(j-1)+i$ with $i,j\in [n]$.
Given $n$-dimensional column vectors $w$ and $v$, we write
$w\odot v$ to denote their Hadamard product. This is an $n$-dimensional vector whose entries $(w\odot v)_i:=w_i v_i$, for all $i\in [n]$.

We denote $I_n$ the $n$-dimensional identity matrix.
Given an $n$-dimensional vector $w$, we denote by ${\rm diag}(w)$ the $n\times n$ diagonal matrix with diagonal entries given by ${\rm diag}(w)_{ii}=w_i$,  $i \in [n]$. If $\bar{w}$ is another $n$-dimensional vector, we denote $w(\bar{w})^{\top}$ the $n\times n$ matrix whose entries are given by $w(\bar{w})^{\top}_{ij}=w_i\bar{w}_j$, $i,j\in [n]$. 
Given an $n\times n$ matrix $B$   we denote by $d_0(B)$ the matrix obtained by setting to zero all diagonal elements of $B$ and keeping the off-diagonal elements unaltered. We denote by ${\rm Diag}(B)=B-d_0(B)$, i.e., the diagonal matrix obtained by setting to zero all off-diagonal elements of $B$ and keeping the diagonal elements of $B$ unaltered.
%
We also denote by ${\rm vec}(B)$ its vectorization, that is  the $n^2$-dimensional column vector obtained by stacking the columns of $B$ one by one. Specifically, for $u \in\{1, \ldots, n^2\}$, ${\rm vec}(B)_u = B_{ij}$, with $i,j \in [n]$ uniquely determined by  $u=n(j-1)+i$. 
With slight abuse of notation, we shall often write $\text{vec}(B)_{ij}$ instead of ${\rm vec}(B)_u$ to refer to the entry $u=n(j-1)+i$ of the vector ${\rm vec}(B)$.
Also, when $B$ is such that $I_n-B$ is invertible, we denote ${\cal G}(B)=(I_n-B)^{-1}$ and ${\cal G}_{\geq 1}(B)=(I_n-B)^{-1}-I_n$. If $I_n-B$ is invertible it holds that ${\cal G}(B)=\sum_{u=0}^{\infty}B^u$ and ${\cal G}_{\geq 1}(B)=\sum_{u=1}^{\infty}B^u$ (see, e.g.,  \cite[Theorem 6.2.8]{Horn1991:topics}). 
Moreover, given two $n\times n$ matrices $B$ and $C$, we denote by $B \otimes C$ their Kronecker product, that is the $n^2\times n^2$ matrix  whose entries are given by $(B \otimes C)_{uv}=B_{j\ell}C_{ih}$, where $h, i, j, \ell \in [n]$ are uniquely determined by $u=n(j-1)+i$ and $v=n(\ell-1)+h$. 
Similar as above, we may also write $(B \otimes C)_{ij,hl}$ instead of $(B \otimes C)_{uv}$ to refer to the entry $(u,v)=(n(j-1)+i,n(\ell-1)+h)$ of the matrix $B \otimes C$. 
Let us recall for later use that $(B\otimes C)(D\otimes E)=(BD\otimes CE)$ for all $n\times n$ matrices $B,C,D$ and $E$. Applying this identity iteratively we deduce that $(B\otimes C)^{u}=(B)^u\otimes (C)^u$ for all $u\geq 0$. 
Recall also that  the Kronecker product and the vectorization are related by the following property: given three matrices $B, C, D$ it holds that 
${\rm vec}(B D C)= (C^{\top}\otimes B){\rm vec}(D)$. 
Finally, the max norm of a matrix $B$ is denoted by $\vertiii{B}_{max}:=\max_{i,j}|B_{ij}|$,  whereas, as already mentioned, the sup norm $\vertiii{B}_{\infty}:=\max_i \sum_{j \in [n]}|B_{ij}|$.  

\subsection{Auxiliary results}
\label{sec:preliminary:results}

In this section, we collect three auxiliary results that collectively show that for the Markov chain model defined in the previous section, first and second order statistics can be explicitly computed from the interaction matrix $A$, for every choice of the interaction graph $G$.
Their proofs are presented in Appendix \ref{sec:appendix:proof:priliminary:results}. 

In the first auxiliary result, we show that the vector of the mean firing activity $\Ex_G[X_0]$ can be explicitly computed from the interaction matrix $A$ and the parameter $\lambda$. 

\begin{proposition}[Firing mean] \label{prop:mean}
Let $[n]$ be the set of neurons, $\mathcal{C}=\{C_1, \ldots,  C_k\}$  the $k$ communities, $G=(\mathcal{C}, \theta, \eta)$ the interaction graph and $\lambda\in[0,1]$ the spontaneous activity. Assume that the corresponding interaction  matrix $A$ and $\lambda$ satisfy~\eqref{eq:wellposed}. 
 Then,
 the  firing mean $\Ex_G[X_0]$ is given by
\[
\Ex_G[X_0] = {\cal G}(A^{\top})\lambda\mathbf{1}_{n}\,.
\]
\end{proposition}

In what follows, we write $\Sigma^{(t)}_G$ to denote the $t$-lagged firing covariance matrix that is defined as
\begin{equation}
\label{def:synchronous:cov}
\Sigma^{(t)}_G:=\Ex_G[X_t(X_0)^{\top}]-\Ex_G[X_t](\Ex_G[X_0])^{\top}, t\in\mathbb{N} \,.
\end{equation}
Along the text, we often refer to the $0$-lagged firing covariance matrix $\Sigma^{(0)}_G$ as the synchronous firing covariance matrix.

The second auxiliary result 
tells us that the synchronous firing covariance matrix $\Sigma^{(0)}_G$
can be computed explicitly from the interaction matrix $A$ and the vector of firing variances, defined as
\begin{equation}
\label{def:vector:variances}
v_G(X_0):={\Ex_G[X_{0}]\odot (\mathbf{1}_n-\Ex_G[X_{0}])}\,.   
\end{equation}
Observe that the vector $v_G(X_0)$ can be computed from Proposition \ref{prop:mean}.

\begin{proposition}\label{prop:cov:vec:form}
Assume that the interaction  matrix $A$ 
and the parameter $\lambda$ satisfy~\eqref{eq:wellposed}. Then,     
the synchronous firing covariance matrix $\Sigma^{(0)}_G$ 
can be written in a vectorized form as
\begin{equation}\label{eq:VecEquation}   
 {\rm vec}\big( \Sigma^{(0)}_G \big) ={\cal G}(L){\rm vec}\big( {\rm diag}(v_G(X_0))\big)
 \,,
\end{equation}  
where $v_G(X_0)$ is the vector of the firing variances defined in \eqref{def:vector:variances} 
and $L$ is a $n^{2} \times n^{2}$ matrix whose entries are defined as: 
\begin{equation*}
    L_{u v} = \begin{cases}
       0 & \text{ if $u=n(i-1)+i$, for $i \in [n]$ }
       \\
        (A^{\top} \otimes  A^{\top})_{u v} & \text{ otherwise}
    \end{cases}
    \,.
\end{equation*}
\end{proposition}

Our final auxiliary result relates the synchronous firing covariance matrix with the $t$-lagged firing covariance matrices  $\Sigma^{(t)}_G$. 
Specifically, we have the following.

\begin{proposition}[Lagged firing covariances]\label{prop:t-laggedCov}
For all $t\geq 1$, the t-lagged and the synchronous firing covariances are related by 
\begin{equation*}
\Sigma^{(t)}_G = (A^{\top})^t \Sigma^{(0)}_G\,.
\end{equation*}

\end{proposition}

As already anticipated at the beginning of the section, the results presented in Propositions \ref{prop:mean}, \ref{prop:cov:vec:form} and \ref{prop:t-laggedCov} allow us to compute explicitly the first and second order statistics of the Markov chain model defined through equations \eqref{eq:cond_ind} and \eqref{eq:firingProb} from the interaction matrix $A$, for every choice of the interaction graph $G$. 
The main goal of the next sections is to provide approximations as $n\to\infty$ for these first and second order statistics, when the interaction graph $G$ is random and distributed as a Stochastic Block Model with $k$ communities.

\section{NFC on random environment}\label{sec:NFC_random}

We begin this section by extending  
the NFC to a random environment, i.e., when the underlying interaction graph is random. Given $[n]$ the set of neurons and $\mathcal{C}=\{C_1, \ldots,  C_k\}$ the $k$ communities, the matrix encoding the pairwise interaction between neurons will be distributed similarly as a Stochastic Block Model (SBM) with the $k$ communities~ \cite{JMLR:v18:16-480}. We denote this matrix by $\Theta$ (to differentiate from the constant case $\theta$ in Section~\ref{sec:NFC}). Specifically,  this means that $(\Theta_{ij})_{i,j\in [n]}$ are independent Bernoulli random variables with parameter $(p_b)_{b \in [k]}$ or $q$ depending on whether the corresponding edge is intra or inter, respectively. In particular, $\Theta_{ij}\sim {\rm Ber}(p_b)$ if $i \overset{b}{\sim} j$ for some $b \in [k]$ and $\Theta_{ij}\sim {\rm Ber}(q)$ if $i \nsim j$.  We also assume that the matrix governing whether an edge is inhibitory/excitatory, here denoted by $H$ (to differentiate from the constant case $\eta$ in Section~\ref{sec:NFC}),  has i.i.d. components, independent from $\Theta$, distributed as 
\begin{equation*}
    H_{ij} = \begin{cases}
        +1  & \text{ with probability } \beta  \\
        -1  & \text{ with probability } (1 - \beta)
    \end{cases}
    \,,
\end{equation*}
where $\beta \in [0,1] \setminus \{1/2\}$. 
Recall that we also assume that the $k$ communities have sizes proportional to the total number of neurons, that is, $|C_b|=\alpha_b n\geq 1$ for every $b \in [k]$ with $\alpha_1 + \ldots + \alpha_k=1$ and $\alpha_{min}:=\min_{b\in [k]}\alpha_b >0$. 
Thus, besides the parameters $\lambda, (\mu_b)_{b \in [k]}, \Out, n$,  the NFC on this random environment also depends on $p=(p_b)_{b \in [k]}$, $q, \beta$ and $(\alpha_b)_{b \in [k]}$. Since in our setting, all parameters but $n$ are kept fixed, 
we shall denote by $\mathbb{P}^{(n)}$ the distribution of the interaction graph and by $\mathbb{E}^{(n)}$ the corresponding expectation, to keep the notation lighter.
Moreover, 
given $\theta \in \{0,1\}^{n\times n}$ and $\eta \in \{-1, +1\}^{n\times n}$, we shall denote 
by  $\mathbb{P}_{(\theta, \eta)}$ the stationary distribution of the process $X=(X_t)_{t\geq 0}$ and by $\phi_{(\theta, \eta), i}$, with $i \in [n]$, the firing probability of $i$,  conditioned on $\Theta=\theta$ and $H=\eta$, i.e.,  conditioned on a realization of the interaction graph.  
We shall instead denote 
by  $\mathbb{P}_{(\Theta, H)}$ the (random) distribution of $X_t$, when we conditioned on the random matrices $\Theta$ and $H$, rather than on specific realizations. The corresponding  expectation will be denoted by $\mathbb{E}_{(\Theta, H)}$. 

In this random scenario, the corresponding interaction matrix (see \eqref{eq:matrixA}) is a random matrix. We shall denote $A_{(\theta,\eta)}$ to indicate the realization of this interaction matrix associated to the realization of the interaction graph $\Theta=\theta$ and $H=\eta$. To avoid clutter, when the realization of the interaction graph is not fixed and there is no risk of confusion, we will write $A$ to denote the corresponding random interaction matrix. We adopt a similar notation for the matrix $L$ (see, Proposition \ref{prop:cov:vec:form}) computed from the interaction matrix $A$.  

Let $\Omega^{adm}$ be the set of all realizations of the random environment for which \eqref{eq:wellposed} holds,  i.e., 

\begin{equation}
\label{def:adm_realizations}
 \Omega^{adm}=\left\{(\theta,\eta)\in \{0,1\}^{n\times n}\times \{-1,1\}^{n\times n} :  \vertiii{A^\top_{(\theta,\eta)}}_{\infty}<\lambda\wedge (1-\lambda) \right\}\,.   
\end{equation}


Note that, on the event $\Omega^{adm}$, we have that $\vertiii{A^\top_{(\theta,\eta)}}_{\infty}<1/2$. In particular, the model is well-defined for all pairs $(\theta,\eta)\in \Omega^{adm}$ in the sense that for these pairs the firing probability $\phi_{(\theta,\eta),i}(x)\in (0,1)$ for all $x\in \{0,1\}^n$ and $i\in [n]$. In order to have a well-defined model also for the pairs $(\theta,\eta)\not\in \Omega^{adm}$, we consider for such pairs the truncated firing probability $(\phi_{(\theta,\eta),i}(x))^+\wedge 1$, where $(\cdot)^+$ denotes the positive part.

In what follows, we define 
\begin{equation}\label{def:M}
M:=\max_{b \in [k]}\{\mu_b\alpha_b p_b+\Out(1-\alpha_b)q\}\,,
\end{equation}
and suppose that the model parameters $\lambda, (\alpha_b)_{b \in [k]},(\mu_b)_{b \in [k]}$, $\mu_{\not\sim}, (p_b)_{b \in [k]}$ and $q$ are chosen in such a way that $M$ satisfies
\begin{equation}\label{eq:wellposed_random_scenario}
M<\lambda\wedge (1-\lambda)\,.    
\end{equation}
Under Condition \eqref{eq:wellposed_random_scenario}, we can prove the following result.
\begin{lemma}\label{lem:random_NFC_welldefined}
Suppose that \eqref{eq:wellposed_random_scenario} is satisfied.
Let $\Omega^{adm}$ be the set of admissible realizations defined in \eqref{def:adm_realizations}.  Then, there exists a positive constant 
$C$ depending on the parameters $\lambda, (\alpha_b)_{b \in [k]},(\mu_b)_{b \in [k]}$, $\mu_{\not\sim}, (p_b)_{b \in [k]}$ and $q$ such that $\mathbb{P}^{(n)}((\Theta,H)\in \Omega^{adm})\geq 1-2ne^{-Cn}$. 
\end{lemma}
Lemma \ref{lem:random_NFC_welldefined} ensures that, whenever  Condition \eqref{eq:wellposed_random_scenario} holds, the random environment $(\Theta,H)$ belongs to the set of admissible realizations $\Omega^{adm}$ with probability converging to $1$  exponentially fast as $n\to\infty$. 
Hence, whenever convenient, we restrict our analysis in what follows to environments lying in $\Omega^{adm}$. 
The proof of Lemma \ref{lem:random_NFC_welldefined} is provided in Appendix \ref{sec:tech:lemmas}.

\subsection{Main results}
\label{sec:main:proba:results}

Before stating our first main  result, let us introduce a couple of definitions. 
Let $W \in \mathbb{R}^{k \times k}$ be the matrix with entries given by 
\begin{equation}\label{def:matrix_G}
W_{ab}:= (2\beta -1)\big(\mu_b \alpha_b p_b \mathbbm{1}_{\{a=b\}} + \Out \alpha_{b} q \mathbbm{1}_{\{b \neq a\}}\big)\,, \quad a, b \in [k]\,.
\end{equation}
 It can be shown that, given $i \in C_a$, then   $W_{ab }=\mathbb{E}^{(n)}\big[(A^\top \mathbf{1}_{C_b})_{i}\big]$, i.e., $W_{a b}$ corresponds to the mean cumulative effect that community $b$ exerts on neuron $i$ belonging to community $C_a$.
Note that, if \eqref{eq:wellposed_random_scenario} holds, then  $\vertiii{W}_{\infty}=\max_{a \in [k]} \sum_{b \in [k]}|W_{a b}|\leq M<(1-\lambda)<1$ and the matrix $\left(I_k - W \right)$ 
is invertible, as the spectral radius $\rho(W)$ of the matrix $W$ satisfies $\rho(W)\leq \vertiii{W}_{\infty}$. 
Thus, whenever \eqref{eq:wellposed_random_scenario} holds,  we denote by $\zeta$ the vector of length $k$ defined as 
\begin{equation}
\label{def:limit:vector:firing:probabilities}
\zeta:= {\cal G}(W)\mathbf{1}_{k}\,. 
\end{equation}

Given the $k$-dimensional vector $\zeta$, we then consider the $n$-dimensional vector $\overline{m}$ defined in \eqref{def:mean:field:approximation:for:the:mean}.
Note that $\overline{m}_i=\lambda \zeta_b$ for all entries $i\in C_b$.
Using that $M<(1-\lambda)$, one can check that the vector $\overline{m}$ satisfies $\|\overline{m}\|_{\infty}<1$.
As we will see in our first main result below, this vector can be seen as an approximation for the random vector of the firing mean  $\mathbb{E}_{(\Theta,H)}[X_0]$, as the size of network $n\to\infty$. Moreover, this approximation holds with high probability and its error is of order $\sqrt{\log(n)/n}$. 
 
\begin{theorem}\label{thm:convergence_mean_version2}
Suppose \eqref{eq:wellposed_random_scenario} is granted and let 
$\overline{m}$ be the vector defined in \eqref{def:mean:field:approximation:for:the:mean}.
%
Then, for every $\delta \in (0,1]$ and every $n\geq k$, it holds that 
\begin{equation*}
\mathbb{P}^{(n)}\bigg(\big\Vert \mathbb{E}_{(\Theta,H)}[X_0] - \overline{m}\big\Vert_{\infty}\geq 2 k \lambda \Vert \zeta\Vert_{\infty} R \sqrt{\frac{2\alpha_{\rm min}}{n}\log\Big(\frac{2 k n}{\delta}\Big)}\bigg)\leq \delta+2n e^{-Cn}\,,
\end{equation*}
where 
$R:=\max_{b \in [k]} |\mu_b|\vee |\Out|$ and $C$ is the positive constant given by Lemma \ref{lem:random_NFC_welldefined}.
\end{theorem}  
The proof of Theorem~\ref{thm:convergence_mean_version2} will be given in Appendix \ref{sec:appendix:asymp:app:mean:variance}.

\medskip 
 Using Theorem \ref{thm:convergence_mean_version2}, 
we can show that the random vector of the firing variances $v_{\Theta,H}(X_0)=\mathbb{E}_{(\Theta,H)}[X_0]\odot(\mathbf{1}_n-\mathbb{E}_{(\Theta,H)}[X_0])$ can be approximated, with high probability, by the vector $\overline{v}$ defined in \eqref{def:mean:field:approx:for:variance}, 
up to an error term that is of order $\sqrt{\log(n)/n}.$  
This is the content of the next result, which is also proved in Appendix \ref{sec:appendix:asymp:app:mean:variance}. 
\begin{corollary}\label{cor:convergence_variance}
Suppose \eqref{eq:wellposed_random_scenario} is granted and let 
$\overline{v}$ be the vector defined in \eqref{def:mean:field:approx:for:variance}.
%
Then, for every $\delta \in (0,1]$ and every $n\geq k$, it holds that 
\begin{equation*}
\mathbb{P}^{(n)}\bigg(\big\Vert v_{(\Theta,H)}(X_0) - \overline{v} \big\Vert_{\infty}\geq 2 k \lambda \Vert \zeta\Vert_{\infty} R \sqrt{\frac{2\alpha_{\rm min}}{n}\log\Big(\frac{2 k n}{\delta}\Big)}\bigg)\leq \delta+2n e^{-Cn}\,,
\end{equation*}
where 
$R$ and $C$ are positive constants defined in Theorem \ref{thm:convergence_mean_version2} and Lemma \ref{lem:random_NFC_welldefined} respectively.
\end{corollary}  
Next, we study the asymptotic behavior of the synchronous covariance matrix $\Sigma^{(0)}_{(\Theta,H)}$. 
With high probability, its diagonal terms are uniformly close to the vector $\overline{v}$, up to an error term of order $\sqrt{\log(n)/n}$. 
This follows from Corollary \ref{cor:convergence_variance} and the fact that the diagonal of $\Sigma^{(0)}_{(\Theta,H)}$ is given by the vector $v_{(\Theta,H)}(X_0)$.
Hence, it remains to understand the asymptotic behavior of the off-diagonal entries of synchronous covariance matrix, that is, of the matrix $d_0(\Sigma^{(0)}_{(\Theta,H)})$.

In what follows, let $\xi$ be the $k$-dimensional vector 
\begin{equation}\label{def:xi:vector}
\xi=(\xi_1,\ldots,\xi_k)\,,\quad
\xi_b:=\lambda\zeta_b\big(1-\lambda\zeta_b\big)\,,\quad b\in[k],
\end{equation}
where $\zeta$ is the vector defined in \eqref{def:limit:vector:firing:probabilities}. Let also $\alpha^{-1}:=(\alpha^{-1}_1,\ldots, \alpha^{-1}_k)$. Given the vectors $\xi$ and $\alpha^{-1}$, we then consider the $k\times k$ matrix $\Psi$, defined as 
\begin{equation}\label{def:Psi}
\Psi:=\sum_{u=1}^{\infty}W^{u}\,{\rm diag}\big(\xi\odot \alpha^{-1}\big)\,\big(W^{\top}\big)^{u}\,.
\end{equation}
Recall that under Condition \eqref{eq:wellposed_random_scenario}, we have
$\rho(W)\leq \vertiii{W}_{\infty}<1$, where $\rho(W)$ denotes the spectral radius of the matrix $W$.
Hence, Lemma \ref{lem:Psi:well:defined} ensures that the matrix $\Psi$ is well-defined and is the unique solution of the Stein equation  
\begin{equation}
\label{def:stein:equation}
\Psi-W\Psi W^{\top}=W\,{\rm diag}\big(\xi\odot \alpha^{-1}\big)\,W^{\top}\,.
\end{equation}
Note that the matrix $\Psi$ is symmetric (this can be verified directly from \ref{def:Psi}).

Now, recall the deterministic matrix $\overline{\Sigma}^{(0)}$ defined in \eqref{def:mean:field:approximation:for:the:simul:cov}, whose entry $\overline{\Sigma}^{(0)}_{ij}$ is given by $n^{-1}\Psi_{ab}$ for all $i\in C_a$ and $j\in C_b$. Note that since the matrix $\Psi$ is symmetric, the matrix $\overline{\Sigma}^{(0)}$ is also symmetric.
In our next main result, we show that, with high probability, the entries of $\Sigma^{(0)}_{(\Theta,H)}$ and $\overline{\Sigma}^{(0)}$ become uniformly close as the size of the network $n\to\infty$. 

\begin{theorem}\label{thm:convergence_cov}
Suppose~\eqref{eq:wellposed_random_scenario} is granted and let $\overline{\Sigma}^{(0)}$ be the matrix defined in \eqref{def:mean:field:approximation:for:the:simul:cov}.
Then there exists a positive constant $K$ independent of $n$ such that for every
$\delta\in(0,1]$ and every $n\geq k$,
\[
\mathbb{P}^{(n)}\bigg(\Bigvertiii{ d_0(\Sigma_{(\Theta,H)}^{(0)})-d_0(\overline{\Sigma}^{(0)})}_{max}
\geq K\sqrt{\frac{\log(n/\delta)}{n^3}}\bigg)\leq\delta+2n e^{-Cn}\,,
\]
where $C$ is the positive constant given by Lemma \ref{lem:random_NFC_welldefined}.
\end{theorem}

Note that, by Corollary \ref{cor:convergence_variance}, the diagonal entries of the synchronous covariance matrix are of order $1$ asymptotically as $n\to \infty$, with high probability. 
As a consequence of Theorem~\ref{thm:convergence_cov}, we see that 
its off-diagonal entries have a different order of magnitude,    
scaling as $n^{-1}$. 
Theorem~\ref{thm:convergence_cov} also tells us that the synchronous covariance matrix inherits the block structure present in the underlying random graph $\Theta$.
The proof of Theorem \ref{thm:convergence_cov} is presented in Appendix \ref{sec:asym:approx:syncronous:cov:matrix}. 




Finally, we focus on the asymptotic behavior of the $t$-lagged covariance matrices $\Sigma^{(t)}_{(\Theta,H)}$, for $t\geq 1$. To describe their asymptotic behavior, we consider the random matrix $\overline{\Sigma}^{(t)}_{(\Theta,H)}$ defined in \eqref{def:gen:Gamma:t}.


\begin{theorem}\label{thm:gen:convergence_cov_lag}
Suppose~\eqref{eq:wellposed_random_scenario} is granted and let $\overline{\Sigma}^{(t)}_{(\Theta,H)}$ be the random matrix defined in~\eqref{def:gen:Gamma:t}, for every integer $t\geq 1$.
Then, there exists a positive constant $K$ independent of $n$ and $t$, such that for every
$\delta\in(0,1]$, every $n\geq k$ and every integer $t\geq 1$,
\[
\mathbb{P}^{(n)}\bigg(
\Bigvertiii{\Sigma^{(t)}_{(\Theta,H)}-\overline{\Sigma}^{(t)}_{(\Theta,H)}}_{max}
\geq K\sqrt{\frac{\log(n/\delta)}{n^3}}\bigg)\leq \delta + 2n e^{-Cn}\,,
\]
where $C$ is the positive constant given by Lemma \ref{lem:random_NFC_welldefined}.
\end{theorem}

Unlike
Theorems~\ref{thm:convergence_mean_version2} and~\ref{thm:convergence_cov}, note that in Theorem \ref{thm:gen:convergence_cov_lag} the
 approximating matrix $\overline{\Sigma}^{(t)}_{(\Theta,H)}$ is itself random, as it depends on the $t$-th power of the interaction matrix transpose $A^{\top}_{(\Theta,H)}$. This follows from the identity
$\Sigma^{(t)}_{(\Theta,H)}=(A^{\top}_{(\Theta,H)})^{t}\Sigma^{(0)}_{(\Theta,H)}$ (Proposition~\ref{prop:t-laggedCov})
and the fact that the diagonal and off-diagonal entries of the synchronous covariance matrix  have different scales. To better understand this, write $\Sigma^{(0)}_{(\Theta,H)}={\rm Diag}(v_{(\Theta,H)}(X_0))+d_0(\Sigma^{(0)}_{(\Theta,H)})$ and note that, on one hand,
the off-diagonal part is of order $n^{-1}$ by Theorem \ref{thm:convergence_cov}, so that $(A^{\top})^{t}$ acting on it may be replaced
by its mean at negligible cost. By doing so, we obtain the deterministic block matrix
$\tfrac1n\sum_{a,b}(W^{t}\Psi)_{ab}\mathbf 1_{C_a}\mathbf 1_{C_b}^{\top}$. On the other hand, the
diagonal part is of order $1$ (see Corollary~\ref{cor:convergence_variance}), and the same
replacement in $(A^{\top})^{t}{\rm diag}(v_{(\Theta,H)}(X_0))$ would lead to an error of order
$\sqrt{\log(n)/n}$,   much larger than the $\sqrt{\log(n)/n^3}$ rate of Theorem~\ref{thm:gen:convergence_cov_lag}.
The proof of Theorem \ref{thm:gen:convergence_cov_lag} is given in Appendix \ref{sec:asym:approx:lagged:cov:matrices}.

\subsection{Consequences for community identifiability}
\label{sec:detection}

The results of Section~\ref{sec:main:proba:results} describe how the community structure of the
interaction network is encoded in the first and second order statistics of the firing 
process, as the size of the network $n\to\infty$.  In this section, we focus on the converse direction, describing how the communities can be identified from (the limit of) these statistics.While implementing community detection algorithms is beyond the scope of this work, these latter results provide the theoretical foundation for their future development within this framework. 

First, we consider the limit of the mean firing vector $\mathbb{E}_{(\Theta,H)}[X_0]$ given by the $n$-dimensional vector $\overline{m}$ defined in \eqref{def:mean:field:approximation:for:the:mean}. 
Recall that the vector $\overline{m}$ is defined in terms of the $k$-dimensional vector $\zeta$ given by \eqref{def:limit:vector:firing:probabilities}. 
Clearly, in the case where all entries of $\zeta$ are distinct, we can correctly identify the $k$ communities of the model from the vector $\overline{m}$, up to permutation.  
On the other extreme, when all entries of $\zeta$ are equal, the vector $\overline{m}$ carries no information about the model communities. In the lemma below, we identify a necessary and sufficient condition, depending only on the parameters $(\alpha_{b})_{b\in [k]}$, $(\mu_{b})_{b\in [k]}$, $(p_{b})_{b\in [k]}$, $\Out$ and $q$, under which all entries of $\zeta$ have the same value. Its proof is presented in Appendix \ref{app:proof:sec:community:ident}.

\begin{lemma}
\label{lemma:identificability:through:the:mean}
Suppose condition \eqref{eq:wellposed_random_scenario} is granted and consider $\zeta= {\cal G}(W)\mathbf{1}_{k}$. Then, all entries of $\zeta$ are equal if and only if 
all entries of $W\mathbf{1}_k$ are equal. In this case, for $b\in [k]$, we have
\[
\zeta_b = \frac{1}{1- (W\mathbf{1}_k)_1}=\frac{1}{1- (2\beta -1)(\mu_1 \alpha_1 p_1 + \Out q (1-\alpha_1))}\,.
\] 
\end{lemma}

\begin{remark}
\label{rmk:1}
One can easily check that the vector $W\mathbf{1}_k$ has all entries equal to each other if and only if 
$\alpha_{b}(\mu_{b}p_{b}-\Out q)=\alpha_{c}(\mu_{c}p_{c}-\Out q)$
for all $b,c\in[k].$ Moreover, if the communities are balanced (i.e., $\alpha_b=1/k$ for all $b \in [k]$) the latter condition reduces to $\mu_b p_b =\mu_a p_a$ for all $\, b \in [k]$.
\end{remark}

Lemma \ref{lemma:identificability:through:the:mean} ensures that when $W\mathbf{1}_k$ has some entries with different values, we certainly must find a pair $(b,c)\in [k]\times [k]$, with $b\neq c$, for which the values of $\zeta_{b}$ and $\zeta_{c}$ are different. In this case, 
by studying the coordinates of $\overline{m}$ we can then partition the set of neurons into at least two different classes. Note that Lemma \ref{lemma:identificability:through:the:mean} does not exclude the possibility of still finding entries of $\zeta$ with equal value. Of course, the communities of these coordinates cannot be distinguished by the vector of means.

In the remaining of this section, we focus on the case in which the value of all entries of the vector $\zeta$ is the same.
According to Lemma \ref{lemma:identificability:through:the:mean} and Remark \ref{rmk:1},
the entries of $\zeta$ have the same value if and only if the parameters $(\alpha_{b})_{b\in [k]}$, $(\mu_{b})_{b\in [k]}$, $(p_{b})_{b\in [k]}$, $\Out$ and $q$ satisfy the following condition:
\begin{equation}
 \label{def:kappa}
 \alpha_{1}(\mu_1p_1-\Out q)=\ldots=\alpha_{k}(\mu_kp_k-\Out q)=:\kappa\,.
\end{equation}
Hence,  we shall assume that $\kappa \neq 0$ in what follows. As mentioned above, in such a case the vector of asymptotic firing means $\overline{m}$ carries no information about the communities.  
Interestingly, going to second order statistics such as the synchronous firing covariance we still find enough information for distinguishing the communities, up to permutation. 
This is a consequence of the proposition below. Before stating it, let us introduce some notation. For every integer $t\geq 0$, define the  matrices  
\begin{equation}\label{def:tlagged:block:bias:matrices}
\overline{S}^{(t)}:=\frac{(2\beta-1)^{t}\kappa^{t}\gamma}{n}\sum_{a\in[k]}\alpha_a^{-1}\,\mathbf{1}_{C_a}\mathbf{1}_{C_a}^{\top}\,
\ \text{and} \  
\overline{B}^{(t)}:=\frac{(2\beta-1)^{t}\big(\gamma(M^{t}-\kappa^{t})+\rho_{\nsim}M^{t}\big)}{n}\,\mathbf{1}_{n}\mathbf{1}_{n}^{\top}\,,
\end{equation}
where the real values $\gamma$ and $\rho_{\nsim}$ are given by 
\begin{equation}\label{def:gamma:rhoout}
\gamma:=((2\beta-1)\kappa)^2(1-(2\beta-1)^2\kappa^2)^{-1} \ \text{and} \ \rho_{\nsim}=\gamma((M/\kappa)^2-1)(1-(2\beta-1)^2M^2)^{-1},     
\end{equation}
with $M$ defined in \eqref{def:M}.
Note that the matrix $\overline{S}^{(t)}$ has a community structure that resembles that of the underlying random graph $\Theta$, whereas the matrix $\overline{B}^{(t)}$ have equal entries. Moreover, the entries of both matrices scale as $n^{-1}$.   

\begin{proposition}
\label{prop:limit:lagged:cov:kappa}
Suppose \eqref{eq:wellposed_random_scenario} and \eqref{def:kappa} are granted.  
For every integer $t\geq 0$, let $\overline{S}^{(t)}$ and $\overline{B}^{(t)}$ be the matrices defined in \eqref{def:tlagged:block:bias:matrices}. 
Then, the matrix $\overline{\Sigma}^{(0)}$ defined in \eqref{def:mean:field:approximation:for:the:simul:cov} can be decomposed as
\[
\overline{\Sigma}^{(0)}=\xi_1\big(\overline{S}^{(0)}+\overline{B}^{(0)}\big),
\]
with $\xi_1$ as defined in \eqref{def:xi:vector}.
Moreover, for every integer $t\geq 1$, the random matrix $\overline{\Sigma}^{(t)}_{\Theta,H}$ defined in \eqref{def:gen:Gamma:t} can be decomposed as
\[
\overline{\Sigma}^{(t)}_{(\Theta,H)}
=\xi_1\Big(\big(A^{\top}_{(\Theta,H)}\big)^{t}+\overline{S}^{(t)}+\overline{B}^{(t)}\Big)\,.
\]
\end{proposition}

Proposition \ref{prop:limit:lagged:cov:kappa} is proved in Appendix \ref{app:proof:sec:community:ident}. 
Note that this proposition tells us that, under Condition \eqref{def:kappa}, the synchronous firing covariance is (asymptotically as $n\to\infty$) smaller between neurons in different communities ($n^{-1}\xi\rho_{\nsim}$) than that between neurons in the same community $C_{b}$ ($n^{-1}\xi(\gamma\alpha^{-1}_{b}+\rho_{\nsim})$) for all $b\in [k]$.
As a consequence, the synchronous firing covariance carries  enough information asymptotically for discriminating the communities $C_{1},\ldots, C_{k}$ up to permutation. If, additionally, the sizes of the communities $\alpha_1,\ldots, \alpha_k$ are known, it also allows us to exactly find all the communities.

Proposition \ref{prop:limit:lagged:cov:kappa} also tells us that the $t$-lagged firing covariance matrix  
can be decomposed (modulo the factor $\xi_1$) into the sum of three matrices: the $t$-th power of the interaction matrix  $A^{\top}$, the signal matrix $\overline{S}^{(t)}$ and the bias matrix $\overline{B}^{(t)}$. Note that interaction matrix $A^{\top}$ is random, whereas the matrices $\overline{S}^{(t)}$ and $\overline{B}^{(t)}$ are deterministic.
Note also that, unlike the bias matrix $\overline{B}^{(t)}$ in which all entries are equal ($\xi_1 n^{-1}(2\beta-1)^t(\gamma (M^t-\kappa^t)+\rho_{\nsim}M^t))$), the signal matrix $\overline{S}^{(t)}$ has a community structure that is inherited from the underlying random graph $\Theta$.  
Specifically, the value of the entry $(i,j) \in C_a\times C_b$ of the matrix $\overline{S}^{(t)}$ 
is
$\xi_1 n^{-1}(2\beta -1)^t\kappa^t\gamma \alpha^{-1}_a$, when $a=b$; whereas when  $a \neq b$, the value is $0$.

In classical community detection problems, the goal is to find the communities from a realization of the adjacency matrix of the underlying random graph model. In our framework, the adjacency matrix counterpart is the interaction matrix $A^{\top}$. Since this matrix is not observed in our setting, it cannot be directly used for finding the communities.
As Theorem \ref{thm:gen:convergence_cov_lag} and Proposition \ref{prop:limit:lagged:cov:kappa} show, the interaction matrix $A^{\top}$ is closely related to the the one-lag matrix $\Sigma_{(\Theta,H)}^{(1)}$, asymptotically as $n\to\infty$.
Moreover, the one-lag covariance matrix inherits the community structure of the random graph $\Theta$ not only through the interaction matrix $A^{\top}$ but also through the signal matrix 
$\overline{S}^{(1)}$. 
This suggests that detecting the communities from the one-lag matrix $\Sigma_{(\Theta,H)}$ may be easier than when relying solely on $A^{\top}$. A rigorous proof of this conjecture is beyond the scope of the paper.

\section*{Acknowledgements}
G.I.'s research was partially supported by  FAPERJ (grant E-26/210.516/2024) and by the National Council for Scientific and Technological Development - CNPq  (grant 300150/2026-1).
G.O. was partially supported by the Serrapilheira Institute (grant number Serra – 2211-42049), FAPERJ (grant E-26/204.532/2024) and CNPq (grant 305591/2026-6).

\appendix

\section{Asymptotic approximation for the mean and variance}
\label{sec:appendix:asymp:app:mean:variance}

In this section, we prove Theorem ~\ref{thm:convergence_mean_version2} and Corollary \ref{cor:convergence_variance}.
We start proving Theorem ~\ref{thm:convergence_mean_version2}. In the proof, we rely on the following lemma which is proven in Appendix \ref{sec:tech:lemmas}.

\begin{lemma}\label{lem:1}
Suppose \eqref{eq:wellposed_random_scenario} is granted
and let $\zeta$ and $\overline{m}$ be the vectors defined in \eqref{def:limit:vector:firing:probabilities} and \eqref{def:mean:field:approximation:for:the:mean}, respectively.
Then, on the event $\{(\Theta, H)\in \Omega^{adm}\}$,  it holds that 
\[
\mathbb{E}_{(\Theta, H)}[X_{0}]= \overline{m}+\lambda \sum_{b \in [k]} \zeta_b{\cal G}(A^{\top})\epsilon^{(b)}\,, 
\]
where $\epsilon^{(b)}:=(A^\top - \mathbb{E}_{p,q,\beta}[A^\top]) \mathbf{1}_{C_b}$ for all $b\in [k]$.
\end{lemma}

By Lemma \ref{lem:random_NFC_welldefined}, the proof of  
Theorem~\ref{thm:convergence_mean_version2} can be reduced to the case the event $\{(\Theta, H)\in \Omega^{adm}\}$ holds.  
Hence, given Lemma \ref{lem:1}, the main step in the proof of Theorem~\ref{thm:convergence_mean_version2} consists in showing that the supremum norm of all random vectors 
$\epsilon^{(b)}$,
$b\in [k]$ is $O(\sqrt{\log(n/\delta)/n})$ with probability at least $1-\delta$. Once this established, we then show that 
the same must hold for the vectors 
${\cal G}(A^{\top})\epsilon^{(b)}$ on $\Omega^{adm}$. 


\begin{proof}[Proof on Theorem~\ref{thm:convergence_mean_version2}]
Throughout the proof, we denote $\mathcal{E}=\{(\Theta,H)\in\Omega^{adm}\}$.
With this notation, observe that  for all $\tau>0$ we can write,
\begin{equation}
\label{eq:decomposition:proba}
\mathbb{P}^{(n)}\big(\|\mathbb{E}_{(\Theta,H)}-\overline{m}\|_{\infty}\geq \tau \big)\leq \mathbb{P}^{(n)}\big({\cal E}^c\big)+\mathbb{P}^{(n)}\big({\cal E}\cap \{\|\mathbb{E}_{(\Theta,H)}-\overline{m}\|_{\infty}\geq \tau\} \big).     
\end{equation}
By Lemma \ref{lem:random_NFC_welldefined}, the first term on the right-hand side of \eqref{eq:decomposition:proba} is bounded from above by $2ne^{-Cn}$. 

Hence, to conclude the proof, it suffices to show that the second term on the right of  \eqref{eq:decomposition:proba} is smaller than $\delta$, for $\tau=2 k \lambda \Vert \zeta\Vert_{\infty} R \sqrt{(2\alpha_{\rm min})/n\log\Big((2 k n)/\delta\Big)}$. This can be done as follows. 
By defining $w_{ij}= \theta_{ij}\eta_{ij} - (2\beta -1)p_b$ if $i \overset{b}{\sim} j$ for some $b \in [k]$, and  $w_{ij}= \theta_{ij}\eta_{ij} - (2\beta -1)q$ if  $i \nsim j$, we obtain that, for every $i \in [n]$, 
\[
\epsilon^{(b)}_i= \sum_{j \in C_b} \frac{w_{ji}}{n} \big(\mu_b \mathbbm{1}_{\{i \in C_b\}} + \Out\mathbbm{1}_{\{i \notin C_b\}}\big)\,. 
\]
Note that, for every fixed $i \in [n]$, the random variables 
$(w_{ji})_{j \in C_b}$ are independent, centered, and $|w_{ji}(\mu_b \mathbbm{1}_{\{i \in C_b\}} + \Out\mathbbm{1}_{\{i \notin C_b\}})|\leq 2 \max_{a \in [k]} |\mu_a|\vee |\Out|$. Therefore, by union bound and Hoeffding's inequality we obtain that, for all $z>0$,
\begin{align*}
&\mathbb{P}^{(n)}(\max_{b \in [k] } \Vert \epsilon^{(b)} \Vert_{\infty}>z)\leq \sum_{b \in [k]} \mathbb{P}^{(n)}(\Vert \epsilon^{(b)} \Vert_{\infty}>z)\leq  \sum_{b \in [k]} \sum_{i\in [n]} \mathbb{P}^{(n)}(|\epsilon^{(b)}_i| >z)
\\& \leq \sum_{\ell \in [b]} \sum_{i\in [n]} \mathbb{P}^{(n)}\Big( \Big|\sum_{j \in C_b} w_{ji} \big(\mu_\ell \mathbbm{1}_{\{i \in C_b\}} + \Out\mathbbm{1}_{\{i \notin C_b\}}\Big)  \Big|>n z \big)
\\
&\leq  \sum_{b \in [k]} \sum_{i\in [n]} 2 \exp{\Big\{-\frac{2z^2 n^2}{4 R^2 |C_b|}\Big\}}\leq 2nk \exp{\Big\{-\frac{z^2 n}{2 R^2 \alpha_{min} }\Big\}}  \,, 
\end{align*}
where $R=\max_{b \in [k]} |\mu_b|\vee |\Out|$ and $\alpha_{ min}= \min_{b \in [k]}\alpha_b>0$. In particular, for every $\delta \in (0,1]$, by taking $z= R \sqrt{\frac{2\alpha_{min}}{n}\log \big( \frac{2nk}{\delta}\big)}$, we conclude that 
\begin{equation}\label{eq:bound_err}
\mathbb{P}^{(n)}\bigg(\max_{b \in [k] } \Vert \epsilon^{(b)} \Vert_{\infty}\leq R \sqrt{\frac{2\alpha_{min}}{n}\log \Big( \frac{2n k}{\delta}\Big)} \bigg)\geq 1-\delta\,.    \end{equation}
  
By Lemma~\ref{lem:1}, on the event $\mathcal{E}$, we have that 
\begin{align*}
\mathbb{E}_{(\Theta,H)}[X_{0}] -  \overline{m}  =  \lambda \sum_{b \in [k]}\zeta_b {\cal G}(A^{\top}) \epsilon^{(b)}
\,.    
\end{align*}
Thus, by the triangle inequality, on the event $\mathcal{E}$, which assure that $\vertiii{A^\top}_{\infty}<1/2$, we obtain
\begin{align}
 &\Big\Vert \mathbb{E}_{(\Theta,H)}[X_{0}] -  \overline{m} \Big\Vert_{\infty} 
 %
 %
 \leq \lambda \sum_{b \in [k]} \big\Vert\zeta_b 
 {\cal G}(A^{\top})\epsilon^{(b)}\big\Vert_{\infty}
\leq \lambda \Vert \zeta \Vert_{\infty} \, \big\Vert \sum_{b \in [k]} 
{\cal G}(A^{\top})\epsilon^{(b)}\big\Vert_{\infty}
\nonumber \\
&
%
%
\leq \lambda \Vert \zeta \Vert_{\infty} \sum_{b \in [k]}  \sum_{u=0}^\infty \big\Vert(A^\top)^u \epsilon^{(b)}\big\Vert_{\infty} 
\leq \lambda \Vert \zeta \Vert_{\infty} \sum_{b \in [k]}  \sum_{u=0}^\infty \vertiii{(A^\top)^u}_{\infty} \Vert \epsilon^{(b)}\Vert_{\infty} 
\nonumber\\
&
\overset{(*)}{\leq}\lambda \Vert \zeta \Vert_{\infty} \sum_{b \in [k]} \Vert \epsilon^{(b)}\Vert_{\infty}   \sum_{u=0}^\infty \vertiii{A^\top}_{\infty}^u 
\leq 2 \lambda \Vert \zeta \Vert_{\infty} k  \max_{b\in [k]}\Vert \epsilon^{(b)}\Vert_{\infty} 
\,,
\label{eq:bound:err:mean:bar:m}
\end{align}
where in $(*)$ we used the sub-multiplicativity of the matrix norm $\vertiii{\cdot}_{\infty}$.
Combining the above bound with \eqref{eq:bound_err}, we 
deduce that the second term on the right of  \eqref{eq:decomposition:proba} is smaller than $\delta$, for $\tau=2 k \lambda \Vert \zeta\Vert_{\infty} R \sqrt{(2\alpha_{\rm min})/n\log\Big((2 k n)/\delta\Big)}$, concluding the proof. 
\end{proof}

Given Theorem ~\ref{thm:convergence_mean_version2}, the proof of Corollary \ref{cor:convergence_variance} is straightforward. 

\begin{proof}[Proof of Corollary \ref{cor:convergence_variance}]
First, use that 
\begin{align*}
\| v_{(\Theta,H)(X_0)}-\overline{m}\odot (\mathbf{1}_n-\overline{m})\|_{\infty}&=\max_{i\in [n]}|(\mathbb{E}_{(\Theta,H)}[X_{i,0}]-\overline{m}_i)(1-\overline{m}_i-\mathbb{E}_{(\Theta,H)}[X_{i,0}])|\\
&\leq \| \mathbb{E}_{(\Theta,H)}[X_0]-\overline{m}\|_{\infty}\,,
\end{align*}
and then use Theorem ~\ref{thm:convergence_mean_version2} to conclude.    
\end{proof}

To conclude this section, we now prove Lemma~\ref{lem:1}. The proof relies on the so-called {\it resolvent identity}, which ensures that   
\begin{equation}\label{resolvent:identity}
B^{-1}-C^{-1}=B^{-1}(C-B)C^{-1}\,,    
\end{equation}
for all invertible $n\times n$ matrices $B$ and $C$.



\begin{proof}[Proof of Lemma~\ref{lem:1}]
First, observe that the matrix $(I_n-A^{\top})$ is invertible on $\Omega^{adm}$ and that the matrix $(I_n-\mathbb{E}^{(n)}[A^{\top}])$ is also invertible since \eqref{eq:wellposed_random_scenario} is granted. Applying the resolvent identity \eqref{resolvent:identity} with $B=(I_n-A^{\top})$ and $C=(I_n-\mathbb{E}^{(n)}[A^{\top}])$ so that $B^{-1}={\cal G}(A^{\top})$, $C^{-1}={\cal G}(\mathbb{E}^{(n)}[A^{\top}])$ and $C-B=A^{\top}-\mathbb{E}^{(n)}[A^{\top}]$, we then deduce that on  $\Omega^{adm}$,
\begin{equation}\label{eq1:proof:lem:1}
({\cal G}(A^{\top})-{\cal G}(\mathbb{E}^{(n)}[A^{\top}]))\mathbf{1}_n={\cal G}(A^{\top})(A^{\top}-\mathbb{E}^{(n)}[A^{\top}]){\cal G}(\mathbb{E}^{(n)}[A^{\top}])\mathbf{1}_n\,.   
\end{equation}

\noindent
{\bf Claim.} The matrix ${\cal G}(\mathbb{E}^{(n)}[A^{\top}])$ satisfies
\begin{equation*}
{\cal G}(\mathbb{E}^{(n)}[A^{\top}])\mathbf{1}_n=\sum_{b\in [k]}\zeta_b\mathbf{1}_{C_b }\,.     
\end{equation*}

Combining \eqref{eq1:proof:lem:1} and the above Claim, we have that on $\Omega^{adm}$, 
\begin{align}\label{eq2:proof:lem:1}
{\cal G}(A^{\top})\mathbf{1}_n&=\sum_{b\in [k]}\zeta_b\mathbf{1}_{C_b }+{\cal G}(A^{\top})(A^{\top}-\mathbb{E}^{(n)}[A^{\top}])\Big(\sum_{b\in [k]}\zeta_b\mathbf{1}_{C_b}\Big)\nonumber\\
&=\sum_{b\in [k]}\zeta_b\mathbf{1}_{C_b }+\sum_{b\in [k]}\zeta_b{\cal G}(A^{\top})(A^{\top}-\mathbb{E}^{(n)}[A^{\top}])\mathbf{1}_{C_b}\nonumber\\
&=\sum_{b\in [k]}\zeta_b\mathbf{1}_{C_b }+\sum_{b\in [k]}\zeta_b{\cal G}(A^{\top})\epsilon^{(b)}\,.
\end{align}
The third equality follows directly from the definition of the ``error'' vectors $\epsilon^{(b)}$. By Proposition~\ref{prop:mean}, \eqref{def:mean:field:approximation:for:the:mean} and \eqref{eq2:proof:lem:1}, the proof follows. Hence, it only remains to prove the claim. This is done below.  

\medskip 
\noindent
{\bf Proof of the claim.} First, observe that  
$(\mathbb{E}^{(n)}[A^{\top}])^{u}\mathbf{1}_{C_a}=\sum_{c\in[k]}(W^{u})_{ca}\mathbf{1}_{C_c}$, for every $u\geq1$ and $a\in [k]$, by Lemma~\ref{lem:EAt:power:1Ca}. Next, 
sum over $a\in[k]$ and use that $\sum_{a\in[k]}\mathbf{1}_{C_a}=\mathbf{1}_n$, to obtain for every integer $u\geq 1$,  
\[
\big(\mathbb{E}^{(n)}[A^{\top}]\big)^{u}\mathbf{1}_n=\sum_{c\in[k]}(W^{u}\mathbf{1}_k)_c\,\mathbf{1}_{C_c}\,.
\]
Since this identity also holds for $u=0$ (both sides are equal to $\mathbf{1}_n$), and
$\sum_{u=0}^{\infty}(W^u\mathbf{1}_k)_c=\zeta_c$ for all $c\in [k]$ by~\eqref{def:limit:vector:firing:probabilities},
we deduce that
\begin{align*}
{\cal G}(\mathbb{E}^{(n)}[A^{\top}])\mathbf{1}_n
&=\sum_{u=0}^{\infty}\big(\mathbb{E}^{(n)}[A^{\top}]\big)^{u}\mathbf{1}_n=\sum_{u=0}^{\infty}\sum_{c\in[k]}(W^{u}\mathbf{1}_k)_c\,\mathbf{1}_{C_c}
\\
&=\sum_{c\in[k]}\Big(\sum_{u=0}^{\infty}(W^{u}\mathbf{1}_k)_c\Big)\mathbf{1}_{C_c}
=\sum_{c\in[k]}\zeta_c\,\mathbf{1}_{C_c}\,,    
\end{align*}
concluding the proof.

\end{proof}

\section{Asymptotic approximation for the synchronous covariance matrix}
\label{sec:asym:approx:syncronous:cov:matrix}

This section is devoted to the proof of Theorem \ref{thm:convergence_cov}. Recall that the $n^2\times n^2$ matrix $L$ is defined as
\begin{equation}\label{def:gen:L}
L_{uv}=\begin{cases}
0 & \text{if } u=n(i-1)+i, \text{ for } i\in[n]\,,\\[2pt]
(A^{\top}\otimes A^{\top})_{uv} & \text{otherwise}\,,
\end{cases}
\end{equation}
or, equivalently, 
\[
L_{ij,h\ell}=\mathbbm{1}_{\{i\neq j\}}A^{\top}_{j\ell}A^{\top}_{ih}
=\mathbbm{1}_{\{i\neq j\}}A_{\ell j}A_{hi}\,,\qquad\text{for all } i,j,\ell,h\in[n]\,.
\]

Hereafter, for $a,b\in[k]$, we denote
\begin{equation}\label{def:gen:Vab}
v^{(a,b)}:={\rm vec}\big(d_0(\mathbf{1}_{C_a}\mathbf{1}_{C_b}^{\top})\big).
\end{equation}
Note that $v^{(a,b)}_{ij}=(d_0(\mathbf{1}_{C_a}\mathbf{1}_{C_b}^{\top}))_{ij}=\mathbbm{1}_{\{i\neq j\}}(\mathbf{1}_{C_a}\mathbf{1}_{C_b}^{\top})_{ij}$,  for all $a,b\in [k]$ and $i,j\in [n]$.

In the proof of Theorem \ref{thm:convergence_cov}, we will use the following result.

\begin{proposition}\label{prop:gen:Psi}
Suppose \eqref{eq:wellposed_random_scenario} is granted. 
Let $\overline{v}$ be the vector defined in \eqref{def:mean:field:approx:for:variance} and let $\overline{\Sigma}^{(0)}$ and $L$ be the matrices defined in \eqref{def:mean:field:approximation:for:the:simul:cov} and \eqref{def:gen:L}, respectively. 
Then, for all $n\ge k$,  
\[
{\cal G}_{\geq 1}\big(\mathbb{E}^{(n)}[L]\big){\rm vec}({\rm diag}(\overline{v}))
=
{\rm vec}\Big(d_0(\overline{\Sigma}^{(0)})\Big)
+\frac{1}{n^2}\sum_{a\in[k]}\sum_{c\in[k]}E_{ac}\,v^{(a,c)}\,,
\]
where the $k\times k$ matrix $E=(E_{ac})_{a,c\in[k]}$ satisfies $\vertiii{E}_{max}\leq K$, for some positive constant $K$ which is independent of $n$.
\end{proposition}

Before proving Proposition \ref{prop:gen:Psi}, we will present the proof of Theorem \ref{thm:convergence_cov}.

\begin{proof}[Proof of Theorem \ref{thm:convergence_cov}] 
 We begin as in the proof of Theorem~\ref{thm:convergence_mean_version2}. Denote  $\mathcal{E}=\{(\Theta,H)\in\Omega^{adm}\}$ and note that for all $\tau>0$ we can write, 
\begin{multline}
\label{eq:decomposition:proba2}
\mathbb{P}^{(n)}\Big(\Bigvertiii{ d_0(\Sigma_{(\Theta,H)}^{(0)})-d_0(\overline{\Sigma}^{(0)})}_{max}\geq \tau \Big)\leq \mathbb{P}^{(n)}\big({\cal E}^c\big)
\\
+\mathbb{P}^{(n)}\Big({\cal E}\cap \Big\{\Bigvertiii{ d_0(\Sigma_{(\Theta,H)}^{(0)})-d_0(\overline{\Sigma}^{(0)})}_{max}\geq \tau\Big\} \Big).     
\end{multline}
By Lemma \ref{lem:random_NFC_welldefined}, the first term on the right-hand side of \eqref{eq:decomposition:proba2} is bounded from above by $2ne^{-Cn}$.
Therefore, to conclude the proof, it remains to show that the second term on the right of  \eqref{eq:decomposition:proba2} is smaller than $\delta$ for $\tau=K\sqrt{\frac{\log(n/\delta)}{n^3}}$, where $K$ is a positive constant independent of $n$. This is done below.

In the rest of the proof, $K_1$ denotes a positive constant independent of $n$ that can change from line to line.
We start observing that from Proposition 
\ref{prop:cov:vec:form}, on the event ${\cal E}$, 
we can write ${\rm vec}\big( \Sigma^{(0)}_{(\Theta,H)} \big) = {\cal G}(L) {\rm vec}\big( {\rm diag}(v_{{\Theta,H}}(X_0))\big)$. 
 Moreover, using that
\begin{align*}
{\rm vec}(d_0(\Sigma^{(0)}_{(\Theta,H)}))&={\rm vec}(\Sigma^{(0)}_{(\Theta,H)}-{\rm diag}(v_{(\Theta,H)}(X_0)))\\
&={\rm vec}(\Sigma^{(0)}_{(\Theta,H)})-{\rm vec}({\rm diag}(v_{(\Theta,H)}(X_0)))
\,,
\end{align*}
we obtain that, on the event ${\cal E}$,  
\begin{align*}
{\rm vec}(d_0(\Sigma^{(0)}_{(\Theta,H)}))&={\cal G}(L) {\rm vec}\big( {\rm diag}(v_{{\Theta,H}}(X_0))\big)-{\rm vec}({\rm diag}(v_{(\Theta,H)}(X_0)))\\
&={\cal G}_{\geq 1}(L) {\rm vec}\big( {\rm diag}(v_{{\Theta,H}}(X_0))\big)\,.
\end{align*}
 
Next, using the definition of $L$, the definition of the vector $\overline{v}$ (recall \eqref{def:mean:field:approx:for:variance})  
and the triangle inequality, we have that
\begin{align*}
&\Big\Vert L\Big(\text{vec}({\rm diag}( v_{(\Theta,H)}(X_0))) -\text{vec}({\rm diag}(\overline{v}))\Big)\Big\Vert_{\infty} 
\\
 &= \max_{i,j \in [n^2]}\Big\vert  
\sum_{(h,\ell)\in [n]^2}L_{ij,h\ell}\Big({\rm vec(diag}(v_{(\Theta,H)}(X_0)))
- \text{vec}({\rm diag}(\overline{v}))\Big)_{h \ell} \Big\vert
\\
&=\max_{i,j \in [n^2]}\Big\vert  \mathbbm{1}_{\{i\neq j\}}\sum_{\ell\in [n]}A_{\ell i}A_{\ell j}\Big(\text{vec}({\rm diag}(v_{(\Theta,H)}(X_0))) - \text{vec}({\rm diag}(\overline{v}))\Big)_{\ell \ell} \Big\vert
\\
&= \max_{i,j \in [n^2]}\Big\vert  \mathbbm{1}_{\{i\neq j\}}\sum_{b\in [k]}\sum_{\ell\in C_b}A_{\ell i}A_{\ell j}\big(\mathbb{E}_{(\Theta,H)}[X_{\ell, 0}] (1-\mathbb{E}_{(\Theta,H)}[X_{\ell, 0}]) - \lambda\zeta_b(1-\lambda\zeta_b) \big) \Big\vert
\\
&\leq \big\Vert \mathbb{E}_{(\Theta,H)}[X_0]-\overline{m}\big\Vert_{\infty} \frac{R}{n}\vertiii{A^{\top}}_{\infty}\,,    
\end{align*}
where $R=\max_{b \in [k]} |\mu_b|\vee |\Out|$.
Combining the preceding inequality with inequalities \eqref{eq:bound:err:mean:bar:m} and \eqref{eq:bound_err}, we obtain that for every $\delta\in (0,1]$,
\begin{equation}
\label{eq0.5:proof:thm:convergence:cov}
\mathbb{P}^{(n)}\Big({\cal E}\cap \Big\{\Big\Vert L\Big(\text{vec}({\rm diag}(v_{(\Theta,H)}(X_0)))-\text{vec}({\rm diag}(\overline{v}))\Big)\Big\Vert_{\infty}\leq K_1\sqrt{n^{-3}\log(n/\delta)}\Big\}\Big)\geq 1-\delta\,.    
\end{equation}

Now, since $\vertiii{L}_{\infty}\leq (\vertiii{A^\top}_{\infty})^2$ and $\vertiii{A^\top}_{\infty}\leq 1/2$ on the event $\cal E$, we also have that 
\begin{multline*}
\big\Vert{\cal G}_{\geq 1}(L)\text{vec}({\rm diag}(v_{(\Theta,H)}(X_0)))-{\cal G}_{\geq 1}(L)\text{vec}({\rm diag}(\overline{v}))\big\Vert_{\infty}\leq \Bigvertiii{\sum_{u=0}^{\infty}L^u}_{\infty}\times \\\Big\Vert L\big(\text{vec}({\rm diag}(v_{(\Theta,H)}(X_0)))-\text{vec}({\rm diag}(\overline{v}))\big)\Big\Vert_{\infty}\leq 4/3\,\Big\Vert L\big(\text{vec}({\rm diag}(v_{(\Theta,H)}(X_0)))-\text{vec}({\rm diag}(\overline{v}))\big)\Big\Vert_{\infty}\,.    
\end{multline*}

Combining the above inequality with \eqref{eq0.5:proof:thm:convergence:cov}, it then follows that 
\begin{equation}
\label{eq0.75:proof:thm:convergence:cov}
\mathbb{P}^{(n)}\Big({\cal E}\cap \Big\{\Vert{\cal G}_{\geq 1}(L)\big(\text{vec}({\rm diag}(v_{(\Theta,H)}(X_0)))-\text{vec}({\rm diag}(\overline{v}))\big)\big\Vert_{\infty}\leq K_1\sqrt{n^{-3}\log(n/\delta)} \Big\}\Big)\geq 1-\delta\,.   \end{equation}

Having established \eqref{eq0.75:proof:thm:convergence:cov}, we now conclude. First, observe that by applying the resolvent identity with $B=I_{n^2}-L$ and $C=I_{n^2}-\mathbb{E}^{(n)}[L]$, we have 
\[
{\cal G}_{\geq 1}(L)\,{\rm vec}({\rm diag}(\overline{v}))
={\cal G}_{\geq 1}\big(\mathbb{E}^{(n)}[L]\big){\rm vec}({\rm diag}(\overline{v}))
+{\cal G}(L)\big(L-\mathbb{E}^{(n)}(L))\big){\cal G}\big(\mathbb{E}^{(n)}[L]\big){\rm vec}({\rm diag}(\overline{v}))\,.
\]
Hence, from above identity, \eqref{eq0.75:proof:thm:convergence:cov} and Proposition \ref{prop:gen:Psi} it suffices to show that for every $\delta \in (0,1]$, on the event $\mathcal E$
\begin{align*}
\Big\| \frac{1}{n^2}\sum_{a,c\in[k]}E_{ac} v^{(a,c)} + &{\cal G}(L)\big(L-\mathbb{E}^{(n)}(L))\big) \Big({\rm vec}({\rm diag}(\overline{v})) +  \frac{1}{n}\sum_{a,c\in [k]}\Psi_{ac}v^{(a,c)} +\frac{1}{n^2}\sum_{a,c\in[k]}E_{ac} v^{(a,c)} \Big) \Big\|_{\infty}
\\
&
\leq K_1\sqrt{\frac{\log(n/\delta)}{n^3}},
\end{align*}
with probability at least $1-\delta$.  
Recalling that $\vertiii{E}_{max}\leq K_1$ the first term on the RHS above can be absorbed into the error term.
Hence, to conclude the proof
it suffices to show that for every $\delta \in (0,1]$, 
\begin{equation}
\label{eq1:proof:thm:convergence:cov}
\mathbb{P}^{(n)}\Bigg({\cal E}\cap \Bigg\{\max_{a,b\in [k]}\big\Vert{\cal G}(L)\epsilon^{(a,b)}\big\Vert_{\infty} \leq K_1\sqrt{\frac{\log(n/\delta)}{n}}\Bigg\}\Bigg)\geq 1-\delta\, ,   
\end{equation}
where $\epsilon^{(a,b)}=(L-\mathbb{E}^{(n)}[L])v^{(a,b)}$ with $v^{(a,b)}$ as defined in \eqref{def:gen:Vab} and that 
\begin{equation}
\label{eq2:proof:thm:convergence:cov}
\mathbb{P}^{(n)}\Bigg({\cal E}\cap\big\Vert{\cal G}(L)(L-\mathbb{E}^{(n)}(L)){\rm vec}({\rm diag}(\overline{v}))\big\Vert_{\infty}\leq K_1\sqrt{\frac{\log(n/\delta)}{n^3}}\Bigg)\geq 1-\delta\, .
\end{equation}
To see that \eqref{eq1:proof:thm:convergence:cov} holds, we use Hoeffding's inequality and argue similarly as in the proof of Theorem \ref{thm:convergence_mean_version2}, to first show that   
$$
\max_{a,b\in [k]}\big\Vert\epsilon^{(a,b)}\big\Vert_{\infty} \leq K_1\sqrt{\frac{\log(n/\delta)}{n}}\,.
$$
Then combine the above inequality with the fact that, on the event ${\cal E}$, it holds that $\vertiii{{\cal G}(L)}_{\infty}\leq \sum_{u=0}^{\infty}(\vertiii{L}_{\infty})^{u}\leq 4/3$ .

 As for the inequality \eqref{eq2:proof:thm:convergence:cov}, we argue as follows.
First, observe that  
\begin{align*}
\big(L-\mathbb{E}^{(n)}(L)\big){\rm vec}({\rm diag}(\overline{v}))_{ij}&=\mathbbm{1}_{\{i\neq j\}}\sum_{a\in [k]}\xi_a\sum_{\ell\in C_{a}}(A_{\ell i}A_{\ell j}-\mathbb{E}^{(n)}[A_{\ell i}A_{\ell j}])\\
&=\mathbbm{1}_{\{i\neq j\}}\sum_{a\in [k]}\xi_a n^{-2}\sum_{\ell\in  C_{a}}\Big(\Theta_{\ell i}H_{\ell i}\Theta_{\ell j}H_{\ell j}-\mathbb{E}^{(n)}[\Theta_{\ell i}H_{\ell i}\Theta_{\ell j}H_{\ell j}]\Big)\,.
\end{align*}
Using Hoeffding's inequality, we can show that with probability at least $1-\delta$, 
$$
\max_{i,j\in[n]: \ i\neq j}\Big|\mathbbm{1}_{\{i\neq j\}}\sum_{a \in [k]}\xi_a n^{-1}\sum_{\ell\in C_{a}}\big(\Theta_{\ell i}H_{\ell i}\Theta_{\ell j}H_{\ell j}-\mathbb{E}^{(n)}[\Theta_{\ell i}H_{\ell i}\Theta_{\ell j}H_{\ell j}]\big)\Big|\leq K_1\sqrt{\frac{\log(n/\delta)}{n}}\,,
$$
so that 
$$
\big\Vert(L-\mathbb{E}^{(n)}(L)){\rm vec}({\rm diag}(\overline{v}))\big\Vert_{\infty} \leq K_1\sqrt{\frac{\log(n/\delta)}{n^3}}\,,
$$
with probability at least $1-\delta$. To conclude, we then proceed as before using that, on the event ${\cal E}$, it holds that $\vertiii{{\cal G}(L)}_{\infty}\leq \sum_{u=0}^{\infty}(\vertiii{L}_{\infty})^{u}\leq 4/3$.

\end{proof}

We now prove Proposition \ref{prop:gen:Psi}.

\begin{proof}[Proof of Proposition \ref{prop:gen:Psi}]
Since ${\rm vec}({\rm diag}(\overline{v}))={\rm vec}({\rm diag}(\sum_{b\in [k]}\xi_b \mathbf{1}_{C_b})) =\sum_{b\in [k]}\xi_b\,{\rm vec}({\rm diag}(\mathbf{1}_{C_b}))$ and ${\cal G}_{\geq 1}(\mathbb{E}^{(n)}[L])=\sum_{u\geq 1}(\mathbb{E}^{(n)}[L])^u$, it follows from item~(iii) of Lemma~\ref{lem:gen:L} that, 
\begin{equation}
\label{eq:1:proof:prop:gen:Psi}
{\cal G}_{\geq 1}\big(\mathbb{E}^{(n)}[L]\big){\rm vec}({\rm diag}(\overline{v}))
=\frac{1}{n}\sum_{u=1}^{\infty}\sum_{a,c\in [k]}\big(H^{(u)}\big)_{ac}v^{(a,c)}
=\frac{1}{n}\sum_{a,c\in [k]}\Big(\sum_{u=1}^{\infty}H^{(u)}\Big)_{ac}v^{(a,c)}\,.
\end{equation}
Set $\Phi:=\sum_{u=1}^{\infty}H^{(u)}$. We first check that $\Phi$ is well-defined. Since $n\alpha_a\geq 1$ for all $a \in [k]$, we have that $0\leq 1-\tfrac{1}{n\alpha_a}\leq 1$ for all $a\in [k]$. From this it follows that 
\[
\vertiii{H^{(u)}-\tfrac{1}{n}{\rm Diag}(H^{(u)})\,{\rm diag}(\alpha^{-1})}_{max}
=\Big(\max_{a\neq b}|H^{(u)}_{ab}|\Big)\vee\Big(\max_{a}\big|H^{(u)}_{aa}\big(1-\tfrac{1}{n\alpha_a}\big)\big|\Big)
\leq\vertiii{H^{(u)}}_{max}\,.
\]
Using the inequalities $\vertiii{BC}_{max}\leq\vertiii{B}_{\infty}\vertiii{C}_{max}$ and
$\vertiii{BC}_{max}\leq\vertiii{B}_{max}\vertiii{C^{\top}}_{\infty}$ together with
$\vertiii{W}_{\infty}\leq M$ and~\eqref{eq:gen:H:rec}, we then obtain
\[
\vertiii{H^{(u+1)}}_{max}\leq\vertiii{W}_{\infty}\,
\vertiii{H^{(u)}-\tfrac{1}{n}{\rm Diag}(H^{(u)})\,{\rm diag}(\alpha^{-1})}_{max}\,\vertiii{W}_{\infty}
\leq M^{2}\,\vertiii{H^{(u)}}_{max}\,.
\]
Iterating the above inequality, we deduce that $\vertiii{H^{(u+1)}}_{max}\leq M^{2u}\vertiii{H^{(1)}}_{max}$. Now, since
$M<1/2$ under Condition \eqref{eq:wellposed_random_scenario} and $\vertiii{H^{(1)}}_{max}=\max_{a,b}\big|\sum_{c}W_{ac}\xi_c W_{cb}\alpha_b^{-1}\big|<\infty$, it follows then that
\[
\vertiii{\Phi}_{max}\leq\sum_{u=1}^{\infty}\vertiii{H^{(u)}}_{max}
\leq\Big(\sum_{u=0}^{\infty}M^{2u}\Big)\vertiii{H^{(1)}}_{max}<\infty\,,
\]
proving that $\Phi$ is well-defined. 

Summing the recursion~\eqref{eq:gen:H:rec} over
$u\geq 1$ and using the definition of the matrix $\Phi$, we see that
\begin{equation}\label{eq:gen:Phi:sylvester}
\Phi=H^{(1)}+W\Big(\Phi-\frac{1}{n}{\rm Diag}(\Phi)\,{\rm diag}(\alpha^{-1})\Big)W^{\top}\,.
\end{equation}
Next, we solve~\eqref{eq:gen:Phi:sylvester} in its vectorized form. Starting from \eqref{eq:gen:Phi:sylvester} and using the linearity of the vec operation together with the identity ${\rm vec}(WBW^{\top})=(W\otimes W){\rm vec}(B)$, we have that  
\[
{\rm vec}(\Phi)={\rm vec}(H^{(1)})+(W\otimes W)\Big[{\rm vec}(\Phi)-\frac{1}{n}{\rm vec}\big({\rm Diag}(\Phi)\,{\rm diag}(\alpha^{-1})\big)\Big]
\,.
\]
Let $Q\in\mathbb{R}^{k^2\times k^2}$ be the matrix with entries
$Q_{ab,cd}=\mathbbm{1}_{\{a=b=c=d\}}\alpha_a^{-1}$, for all $a,b,c,d\in[k]$ and note that for all $a,b\in[k]$,
\[
\big(Q\,{\rm vec}(\Phi)\big)_{ab}=\mathbbm{1}_{\{a=b\}}\alpha_a^{-1}\,{\rm vec}(\Phi)_{aa}
={\rm vec}\big({\rm Diag}(\Phi)\,{\rm diag}(\alpha^{-1})\big)_{ab}\,,
\]
i.e., ${\rm vec}({\rm Diag}(\Phi){\rm diag}(\alpha^{-1}))=Q\,{\rm vec}(\Phi)$. Consequently, it follows that
\[
\Big(I_{k^2}-(W\otimes W)\big(I_{k^2}-\tfrac{1}{n}Q\big)\Big){\rm vec}(\Phi)={\rm vec}(H^{(1)})\,.
\]
Observe that the matrix $I_{k^2}-(W\otimes W)(I_{k^2}-\tfrac{1}{n}Q)$ is invertible. Indeed,
\begin{align*}
\rho\big((W\otimes W)\big(I_{k^2}-\tfrac{1}{n}Q\big)\big)&\leq \vertiii{(W\otimes W)\big(I_{k^2}-\tfrac{1}{n}Q\big)}_{\infty}\\
&\leq\vertiii{W\otimes W}_{\infty}\,\vertiii{I_{k^2}-\tfrac{1}{n}Q}_{\infty}\\
&=\Big(\max_{a,b}\big|\textstyle\sum_{d}W_{bd}\big|\big|\sum_{c}W_{ac}\big|\Big)
\Big(1\vee\max_{a}\big|1-\tfrac{1}{n\alpha_a}\big|\Big)
\\
&\leq\vertiii{W}_{\infty}^{2}\leq M^{2}<1/4\,.
\end{align*}
where last inequality we have used Condition \eqref{eq:wellposed_random_scenario}.

As a consequence, we deduce that
\[
{\rm vec}(\Phi)=\Big(I_{k^2}-(W\otimes W)\big(I_{k^2}-\tfrac{1}{n}Q\big)\Big)^{-1}{\rm vec}(H^{(1)})\,.
\]
Now, using the resolvent identity \eqref{resolvent:identity} with $B=I_{k^2}-(W\otimes W)\big(I_{k^2}-\tfrac{1}{n}Q\big)$ and $C=I_{k^2}-(W\otimes W)$, we see that
\begin{multline*}
\Big(I_{k^2}-(W\otimes W)\big(I_{k^2}-\tfrac{1}{n}Q\big)\Big)^{-1}
=\big(I_{k^2}-W\otimes W\big)^{-1}\\
-\frac{1}{n}\Big(I_{k^2}-(W\otimes W)\big(I_{k^2}-\tfrac{1}{n}Q\big)\Big)^{-1}
(W\otimes W)\,Q\,\big(I_{k^2}-W\otimes W\big)^{-1}\,,
\end{multline*}
which, in turn, allow us to conclude that 
\begin{equation}
\label{eq:gen:Phi:sylvester:vec:form}
{\rm vec}(\Phi)={\rm vec}(\Psi)+\tfrac{1}{n}{\rm vec}(E),   
\end{equation}
where $\Psi:={\rm vec}^{-1}\Big(\big(I_{k^2}-W\otimes W\big)^{-1}{\rm vec}(H^{(1)})\Big)$ and 
\[
E:={\rm vec}^{-1}\Big(-\Big(I_{k^2}-(W\otimes W)\big(I_{k^2}-\tfrac{1}{n}Q\big)\Big)^{-1}
(W\otimes W)\,Q\,\big(I_{k^2}-W\otimes W\big)^{-1}{\rm vec}(H^{(1)})\Big)\,.
\]
Let us check that $\Psi$ satisfies \eqref{def:Psi}.
Using first the identity $(B\otimes C)^{u}=B^u\otimes C^u$ which is valid for all $u\geq 0$, and then the identity ${\rm vec}(BDC)=(C^{\top}\otimes B)\,{\rm vec}(D)$, we have
\begin{align*}
{\rm vec}(\Psi)=\big(I_{k^2}-W\otimes W\big)^{-1}{\rm vec}(H^{(1)})
&=\sum_{u=0}^{\infty}(W\otimes W)^{u}{\rm vec}(H^{(1)})
=\sum_{u=0}^{\infty}\big(W^{u}\otimes W^{u}\big){\rm vec}(H^{(1)})\\
&=\sum_{u=0}^{\infty}{\rm vec}\big(W^{u}H^{(1)}(W^{\top})^{u}\big)
={\rm vec}\Big(\sum_{u=0}^{\infty}W^{u}H^{(1)}(W^{\top})^{u}\Big)\,.
\end{align*}
Now by \eqref{eq:gen:commute}, we have  
\[
H^{(1)}=W{\rm diag}(\xi)W{\rm diag}(\alpha^{-1})
=W{\rm diag}(\xi){\rm diag}(\alpha^{-1})W^{\top}
=W{\rm diag}(\xi\odot \alpha^{-1})W^{\top},
\]
so that
$W^{u}H^{(1)}(W^{\top})^{u}=W^{u+1}{\rm diag}(\xi\odot \alpha^{-1})(W^{\top})^{u+1}$. As a consequence, we obtain that 
 \[
{\rm vec}(\Psi)={\rm vec}\Big(\sum_{u=0}^{\infty}W^{u}H^{(1)}(W^{\top})^{u}\Big)={\rm vec}\Big(\sum_{u=0}^{\infty}W^{u+1}\,{\rm diag}\big(\xi\odot \alpha^{-1}\big)\,\big(W^{\top}\big)^{u+1}\Big),\
 \]
establishing \eqref{def:Psi}.

Finally, using that $\vertiii{{\cal G}(D)}_{\infty}\leq(1-\vertiii{D}_{\infty})^{-1}$ for any matrix $D$ such that $\vertiii{D}_{\infty}<1$,
together with $\vertiii{W\otimes W}_{\infty}=\vertiii{W}_{\infty}^2$ and
$\vertiii{Q}_{\infty}=\alpha_{min}^{-1}$, one can check that
\[
\vertiii{E}_{max}=\Vert {\rm vec}(E)\Vert_{\infty}
\leq\frac{\vertiii{Q}_{\infty}\,\vertiii{W}_{\infty}^{2}}{\big(1-\vertiii{W}_{\infty}^{2}\big)^{2}}\,
\vertiii{H^{(1)}}_{max}
\leq\frac{\vertiii{W}_{\infty}^{4}}{1-\vertiii{W}_{\infty}^{2}}\cdot\frac{1}{\alpha_{min}}\cdot\max_{a\in[k]}\frac{\xi_a}{\alpha_a}
=:K
\]
Clearly $K<\infty$ because $\vertiii{W}_{\infty}\leq M<1/2$ by Condition  \eqref{eq:wellposed_random_scenario}. 

Inverting the vectorization in equation \eqref{eq:gen:Phi:sylvester:vec:form} yields
$\Phi=\Psi+\tfrac{1}{n}E$ with $\vertiii{E}_{max}\leq K$, and the result follows from \eqref{eq:1:proof:prop:gen:Psi} and the fact that 
$\frac{1}{n}\sum_{a,c\in [k]}\Psi_{ac}v^{(a,c)}={\rm vec}(d_0(\overline{\Sigma}^{(0)})).$
\end{proof}

\section{Asymptotic approximation for the lagged covariance matrices}
\label{sec:asym:approx:lagged:cov:matrices}

Here, we prove Theorem \ref{thm:gen:convergence_cov_lag}.

\begin{proof}[Proof of Theorem  \ref{thm:gen:convergence_cov_lag}]
Throughout the proof, we denote $\mathcal{E}=\{(\Theta,H)\in\Omega^{adm}\}$. 
Fix an integer $t\geq1$. As in the proof of Theorem~\ref{thm:convergence_mean_version2}, for
every $\tau>0$, note that we can write
\begin{equation}\label{eq:lag:reduction}
\mathbb{P}^{(n)}\Big(\Bigvertiii{\Sigma^{(t)}_{(\Theta,H)}-\overline{\Sigma}^{(t)}_{(\Theta,H)}}_{max}\geq\tau\Big)
\leq\mathbb{P}^{(n)}(\mathcal{E}^c)
+\mathbb{P}^{(n)}\Big(\mathcal{E}\cap\Big\{\Bigvertiii{\Sigma^{(t)}_{(\Theta,H)}-\overline{\Sigma}^{(t)}_{(\Theta,H)}}_{max}\geq\tau\Big\}\Big)\,.
\end{equation}
By Lemma~\ref{lem:random_NFC_welldefined}, the first term is at most $2ne^{-Cn}$. Therefore, it
suffices to show that the second term is at most $\delta$, for $\tau=K\sqrt{\log(n/\delta)/n^{3}}$ where $K$ is a suitable positive independent of $n$ and $t$.

By Proposition~\ref{prop:t-laggedCov}, we have that $\Sigma^{(t)}_{(\Theta,H)}=(A^{\top})^{t}\Sigma^{(0)}_{(\Theta,H)}$. Therefore, 
decomposing $\Sigma^{(0)}_{(\Theta,H)}={\rm diag}(v_{(\Theta,H)}(X_0))+d_0(\Sigma^{(0)}_{(\Theta,H)})$
and recalling the definition of $\overline{\Sigma}^{(t)}_{(\Theta,H)}$ (see \eqref{def:gen:Gamma:t}), we can write
\[
\Sigma^{(t)}_{(\Theta,H)}-\overline{\Sigma}^{(t)}_{(\Theta,H)}=R_1+R_2+R_3+R_4\,,
\]
where
\begin{align*}
R_1&:=(A^{\top})^{t}\big(d_0(\Sigma^{(0)}_{(\Theta,H)})-d_0(\overline{\Sigma}^{(0)})\big)\,,
&R_3&:=\big((A^{\top})^{t}-(\mathbb{E}[A^{\top}])^{t}\big)d_0(\overline{\Sigma}^{(0)})\,,\\
R_2&:=(A^{\top})^{t}\big({\rm diag}(v_{(\Theta,H)}(X_0))-{\rm diag}(\overline{v})\big)\,,
&R_4&:=(\mathbb{E}[A^{\top}])^{t}d_0(\overline{\Sigma}^{(0)})-\frac{1}{n}\sum_{a,b\in[k]}(W^{t}\Psi)_{ab}\mathbf{1}_{C_a}\mathbf{1}_{C_b}^{\top}\,.
\end{align*}
Below, we upper bound
each term on the event ${\cal E}$. In what follows,
$K_1$ denotes a positive constant independent of $n$ and $t$ that may change from line to line.

\medskip\noindent
\emph{Term $R_1$.} Using that $\vertiii{D_1D_2}_{max}\leq\vertiii{D_1}_{\infty}\vertiii{D_2}_{max}$
and that $\vertiii{(A^{\top})^{t}}_{\infty}\leq (\vertiii{(A^{\top}}_{\infty})^{t}<1$ on the event ${\cal E}$, we have   
\[
\vertiii{R_1}_{max}\leq\vertiii{(A^{\top})^{t}}_{\infty}
\vertiii{d_0(\Sigma^{(0)}_{(\Theta,H)})-d_0(\overline{\Sigma}^{(0)})}_{max}
\leq\vertiii{d_0(\Sigma^{(0)}_{(\Theta,H)})-d_0(\overline{\Sigma}^{(0)}))}_{max}\,,
\]
on the event ${\cal E}$.
By the proof of Theorem \ref{thm:convergence_cov} (see \eqref{eq:decomposition:proba2}), we can then conclude from the above inequality that 
\[
\mathbb{P}^{(n)}\big({\cal E}\cap \{\vertiii{R_1}_{max}\geq K_1\sqrt{\log(n/\delta)/n^3} \}\big)\leq \delta.
\]

\medskip\noindent
\emph{Term $R_2$.} Using the inequalities
$\vertiii{(A^{\top})^{t}B}_{max}\leq\vertiii{(A^{\top})^{t-1}}_{\infty}\vertiii{A^{\top}B}_{max}$
and  $\vertiii{A^{\top}B}_{max}\leq\vertiii{A^{\top}}_{max}\vertiii{B}_{\infty}$ both with
$B={\rm diag}(v_{(\Theta,H)}(X_0))-{\rm diag}(\overline{v})$, we deduce that
\[
\vertiii{R_2}_{max}\leq \vertiii{(A^{\top})^{t-1}}_{\infty} \vertiii{A^{\top}}_{max} \big\Vert v_{(\Theta,H)}(X_0)-\overline{v}\big\Vert_{\infty}.  
\]
Since $\vertiii{A^{\top}}_{max}\leq n^{-1}(\max_{b\in[k]}\mu_b\vee\Out)$ and
$\vertiii{(A^{\top})^{t}}_{\infty}\leq (\vertiii{(A^{\top}}_{\infty})^{t}<1$ on the event ${\cal E}$, we then obtain
\[
\vertiii{R_2}_{max}\leq\frac{\max_{b\in[k]}\mu_b\vee\Out}{n}
\big\Vert v_{(\Theta,H)}(X_0)-\overline{v}\big\Vert_{\infty}\,.
\]
Finally, since 
$\Vert v_{(\Theta,H)}(X_0)-\overline{v}\Vert_{\infty}\leq\big\Vert\mathbb{E}_{(\Theta,H)}[X_0]-\overline{m}\big\Vert_{\infty}$, we can combine the above inequality with inequalities \eqref{eq:bound:err:mean:bar:m} and \eqref{eq:bound_err} to deduce that
\[
\mathbb{P}^{(n)}\big({\cal E}\cap \{\vertiii{R_2}_{max}\geq K_1\sqrt{\log(n/\delta)/n^3} \}\big)\leq \delta.
\]

\medskip\noindent
\emph{Term $R_3$.} Let $\Delta:=A^{\top}-\mathbb{E}^{(n)}[A^{\top}]$ and recall the
identity (which can be easily proved by induction on $t$),
\[
(A^{\top})^{t}-(\mathbb{E}^{(n)}[A^{\top}])^{t}=\sum_{s=0}^{t-1}(A^{\top})^{t-1-s}\,\Delta\,(\mathbb{E}^{(n)}[A^{\top}])^{s}\,.
\]
Using $\vertiii{D_1D_2}_{max}\leq\vertiii{D_1}_{\infty}\vertiii{D_2}_{max}$, the
sub-multiplicativity of $\vertiii{\cdot}_{\infty}$, the bound $\vertiii{A^{\top}}_{\infty}<1$ on the event ${\cal E}$,
and the fact that $\vertiii{\mathbb{E}[A^{\top}]}_{\infty}=\vertiii{W}_{\infty}\leq M\leq 1/2$ (the last inequality holds because Condition \eqref{eq:wellposed_random_scenario} is granted), we then obtain that 
\begin{align*}
\vertiii{R_3}_{max}&\leq\vertiii{(A^{\top})^{t}-(\mathbb{E}[A^{\top}])^{t}}_{\infty}\,\vertiii{d_0(\overline{\Sigma}^{(0)})}_{max}\\
&\leq\vertiii{\Delta}_{\infty}\Big(\sum_{s=0}^{t-1}2^{-s}\Big)\vertiii{d_0(\overline{\Sigma}^{(0)})}_{max}
\leq 2\vertiii{\Delta}_{\infty}\,\vertiii{d_0(\overline{\Sigma}^{(0)})}_{max}\,.    
\end{align*}
Finally, using that $\vertiii{d_0(\overline{\Sigma}^{(0)})}_{max}\leq n^{-1}\vertiii{\Psi}_{max}$ and that $\vertiii{\Delta}_{\infty}=\vertiii{A^\top- \mathbb{E}^{(n)}[A^\top]}_{\infty}\leq K_1\sqrt{\log (n/\delta)/n}$ with probability at least $1-\delta$ (which can be proved similarly to what we did in the proof of Theorem~\ref{thm:convergence_mean_version2}), we conclude that
\[
\mathbb{P}^{(n)}\big({\cal E}\cap \{\vertiii{R_3}_{max}\geq K_1\sqrt{\log(n/\delta)/n^3} \}\big)\leq \delta.
\]

\medskip\noindent
\emph{Term $R_4$.} Note that $R_4$ is a deterministic matrix. By Lemma \ref{lem:EAt:power:1Ca}, we know that 
$(\mathbb{E}^{(n)}[A^{\top}])^{t}\mathbf{1}_{C_a}=\sum_{c\in[k]}(W^{t})_{ca}\mathbf{1}_{C_c}$ for every $t\geq 1$. 
Writing
$d_0(\mathbf{1}_{C_a}\mathbf{1}_{C_b}^{\top})=\mathbf{1}_{C_a}\mathbf{1}_{C_b}^{\top}-\mathbbm{1}_{\{a=b\}}{\rm diag}(\mathbf{1}_{C_a})$ and also using that 
$B\mathbf{1}_{C_a}\mathbf{1}_{C_b}^{\top}=(B\mathbf{1}_{C_a})\mathbf{1}_{C_b}^{\top}$ for any $n\times n$ matrix $B$, 
we obtain
\begin{align*}
(\mathbb{E}[A^{\top}])^{t}d_0(\overline{\Sigma}^{(0)})
&=\frac{1}{n}\sum_{a,b\in[k]}\Psi_{ab}\big((\mathbb{E}^{(n)}[A^{\top}])^{t}\mathbf{1}_{C_a}\big)\mathbf{1}_{C_b}^{\top}-\frac{1}{n}\sum_{a\in[k]}\Psi_{aa}\,(\mathbb{E}^{(n)}[A^{\top}])^{t}{\rm diag}(\mathbf{1}_{C_a}) 
\\
=&\frac{1}{n}\sum_{a,b\in[k]}\Psi_{ab}\Big(\sum_{c\in[k]}(W^{t})_{ca}\mathbf{1}_{C_c}\Big)\mathbf{1}_{C_b}^{\top}
-\frac{1}{n}\sum_{a\in[k]}\Psi_{aa}\,(\mathbb{E}^{(n)}[A^{\top}])^{t}{\rm diag}(\mathbf{1}_{C_a})\,.
\end{align*}
Since $\sum_{a}(W^{t})_{ca}\Psi_{ab}=(W^{t}\Psi)_{cb}$, the first sum is equal to
$\frac{1}{n}\sum_{c,b\in[k]}(W^{t}\Psi)_{cb}\mathbf{1}_{C_c}\mathbf{1}_{C_b}^{\top}$, and therefore
\[
R_4=-\frac{1}{n}\sum_{a\in[k]}\Psi_{aa}\,(\mathbb{E}[A^{\top}])^{t}{\rm diag}(\mathbf{1}_{C_a})\,.
\]
Combining the above identity with the inequalities 
$\vertiii{(\mathbb{E}[A^{\top}])^{t}{\rm diag}(\mathbf{1}_{C_a})}_{max}\leq\vertiii{(\mathbb{E}[A^{\top}])^{t}}_{max}
\leq\vertiii{\mathbb{E}[A^{\top}]}_{\infty}^{t-1}\vertiii{\mathbb{E}[A^{\top}]}_{max}$ and $\vertiii{\mathbb{E}[A^{\top}]}_{max}=n^{-1}\vertiii{\mathcal{W}}_{max}$,
it follows that
\[
\vertiii{R_4}_{max}\leq\frac{1}{n}\Big(\sum_{a\in[k]}|\Psi_{aa}|\Big)\vertiii{\mathbb{E}[A^{\top}]}_{\infty}^{t-1}\frac{\vertiii{\mathcal{W}}_{max}}{n}\leq K_1\frac{\vertiii{\mathbb{E}[A^{\top}]}_{\infty}^{t-1}}{n^{2}}\,.
\]
Finally, using that $\vertiii{\mathbb{E}[A^{\top}]}_{\infty}^{t-1}<1$ on the event ${\cal E}$, we deduce that $\vertiii{R_4}_{max}\leq K_1n^{-2}$
which can be absorbed in the error $K_1\sqrt{\log(n/\delta)/n^3}$.

\medskip
Combining the bounds on error terms $R_1$, $R_2$, $R_3$ and $R_4$ via union bound with $\delta/4$ replacing $\delta$ shows that that the second term of \eqref{eq:lag:reduction} is at most $\delta$, for $\tau=K\sqrt{\log(n/\delta)/n^{3}}$ where $K$ is a suitable positive independent of $n$ and $t$, concluding the proof.
\end{proof}

\section{On the mean of the matrices \texorpdfstring{$A$}{A} and \texorpdfstring{$L$}{L}}
\label{sec:mean:A:and:L}

In this section, we collect some auxiliary results concerning the mean of the interaction matrix $A$ and of the matrix $L$ defined in \eqref{def:gen:L}.

Before stating the first auxiliary result, let $\mathcal{W}$ be the  $k\times k$ matrix defined by
\begin{equation}\label{def:gen:Wbar}
\mathcal{W}_{ab}:=(2\beta-1)\big(\mu_a p_a\,\mathbbm{1}_{\{a=b\}}+\Out q\,\mathbbm{1}_{\{a\neq b\}}\big)\,,\quad a,b\in[k]\,.
\end{equation}
Observe that the matrix $\mathcal{W}$ is symmetric, i.e.\ $\mathcal{W}=\mathcal{W}^{\top}$, and
that it is related to the matrix $W$ defined in~\eqref{def:matrix_G} by
\begin{equation}\label{eq:gen:W:Wbar}
W=\mathcal{W}\,{\rm diag}(\alpha)\, \text{with} \ 
\alpha=(\alpha_1,\ldots,\alpha_k)\,.
\end{equation}
As a consequence of~\eqref{eq:gen:W:Wbar} and the symmetry of $\mathcal{W}$, one can easily check that
\begin{equation}\label{eq:gen:commute}
W\,{\rm diag}(\alpha^{-1})=\mathcal{W}={\rm diag}(\alpha^{-1})\,W^{\top}\,,
\end{equation}
Iterating \eqref{eq:gen:commute}, we deduce that $(W)^{u}{\rm diag}(\alpha^{-1})={\rm diag}(\alpha^{-1})(W^{\top})^{u}$
for all $u\geq 0$.

The first result of this section reads as follows. 

\begin{lemma}\label{lem:gen:EAtranspose}
Let $A$ be the interaction matrix defined in \eqref{eq:matrixA}
and $\mathcal{W}$ the matrix defined in \eqref{def:gen:Wbar}. Then we can write 
\begin{equation}\label{eq:gen:EA:Wbar}
\mathbb{E}[A^{\top}]=\mathbb{E}[A]=\frac{1}{n}\sum_{a,b\in[k]}\mathcal{W}_{ab}\,\mathbf{1}_{C_a}\mathbf{1}_{C_b}^{\top}\,.
\end{equation}
Moreover, for all $a,b\in[k]$,
\begin{equation}\label{eq:gen:lem:EAtranspose:1}
\mathbb{E}[A^{\top}]\,\mathbf{1}_{C_a}\mathbf{1}_{C_b}^{\top}
=\sum_{c\in[k]}W_{ca}\,\mathbf{1}_{C_c}\mathbf{1}_{C_b}^{\top}\,.
\end{equation}
Furthermore,
\begin{equation}\label{eq:gen:lem:EAtranspose:2}
\mathbb{E}[A^{\top}]\,\mathbf{1}_{n}\mathbf{1}_{n}^{\top}
=\sum_{c\in[k]}(W\mathbf{1}_k)_c\,\mathbf{1}_{C_c}\mathbf{1}_{n}^{\top}\,.
\end{equation}
\end{lemma}


\begin{proof}[Proof of Lemma~\ref{lem:gen:EAtranspose}]
By the definition of interaction matrix $A$ and the
independence of the environment, for all $i,j\in[n]$, we have
\[
\mathbb{E}[A_{ij}]=\frac{(2\beta-1)}{n}\Big(\sum_{a\in[k]}\mu_a p_a\,
\mathbbm{1}_{\{i\overset{a}{\sim}j\}}+\Out q\,\mathbbm{1}_{\{i\nsim j\}}\Big)\,.
\]
Since the right-hand side is symmetric in $(i,j)$, we have
$\mathbb{E}[A_{ij}]=\mathbb{E}[A_{ji}]$ for all $i,j\in[n]$ which, in turn, implies the first equality in \eqref{eq:gen:EA:Wbar}.
Moreover, the expression above for $\mathbb{E}[A_{ij}]$ and \eqref{def:gen:Wbar} imply that
\[
\mathbb{E}[A]=\frac{(2\beta-1)}{n}\Big(\sum_{a\in[k]}\mu_a p_a\,
\mathbf{1}_{C_a}\mathbf{1}_{C_a}^{\top}
+\Out q\sum_{a\in[k]}\sum_{\substack{b\in[k]\\ b\neq a}}\mathbf{1}_{C_a}\mathbf{1}_{C_b}^{\top}\Big)=\frac{1}{n}\sum_{a,b\in[k]}\mathcal{W}_{ab}\,\mathbf{1}_{C_a}\mathbf{1}_{C_b}^{\top}\,,
\]
proving the second equality in \eqref{eq:gen:EA:Wbar}.

Next, we will prove \eqref{eq:gen:lem:EAtranspose:1}. For that end, fix $a,b\in[k]$ and observe that by \eqref{eq:gen:EA:Wbar} and \eqref{eq:gen:W:Wbar}, we have
\begin{align*}
\mathbb{E}[A^{\top}]\,\mathbf{1}_{C_a}\mathbf{1}_{C_b}^{\top}
&=\Big(\frac{1}{n}\sum_{c,d\in [k]}\mathcal{W}_{cd}\mathbf{1}_{C_c}\mathbf{1}_{C_d}^{\top}\Big)\mathbf{1}_{C_a}\mathbf{1}_{C_b}^{\top}\\
&=\frac{1}{n}\sum_{c,d\in [k]}\mathcal{W}_{cd}\Big(\mathbf{1}_{C_c}\mathbf{1}_{C_d}^{\top}\mathbf{1}_{C_a}\mathbf{1}_{C_b}^{\top}\Big)\\
&=\sum_{c\in [k]}\mathcal{W}_{ca}\frac{|C_a|}{n}\mathbf{1}_{C_c}\mathbf{1}_{C_b}^{\top}=\sum_{c\in [k]}W_{ca}\mathbf{1}_{C_c}\mathbf{1}_{C_b}^{\top}\,,
\end{align*}
proving \eqref{eq:gen:lem:EAtranspose:1}.
Finally, using $\mathbf{1}_n\mathbf{1}_n^{\top}=\sum_{a\in[k]}\sum_{b\in[k]}\mathbf{1}_{C_a}\mathbf{1}_{C_b}^{\top}$
together with~\eqref{eq:gen:lem:EAtranspose:1}, we obtain that
\begin{align*}
\mathbb{E}[A^{\top}]\,\mathbf{1}_{n}\mathbf{1}_{n}^{\top}
&=\sum_{a\in[k]}\sum_{b\in[k]}\sum_{c\in[k]}W_{ca}\,\mathbf{1}_{C_c}\mathbf{1}_{C_b}^{\top}
=\sum_{a\in[k]}\sum_{c\in[k]}W_{ca}\,\mathbf{1}_{C_c}\Big(\sum_{b\in[k]}\mathbf{1}_{C_b}\Big)^{\top}\\
&=\sum_{c\in[k]}\Big(\sum_{a\in[k]}W_{ca}\Big)\mathbf{1}_{C_c}\mathbf{1}_{n}^{\top}
=\sum_{c\in[k]}(W\mathbf{1}_k)_c\,\mathbf{1}_{C_c}\mathbf{1}_{n}^{\top}\,,
\end{align*}
establishing \eqref{eq:gen:lem:EAtranspose:2}, concluding the proof. 
\end{proof}

Our second result is the following.

\begin{lemma}\label{lem:EAt:power:1Ca}
For every $a\in[k]$ and every integer $u\geq1$,
\[
\big(\mathbb{E}[A^{\top}]\big)^{u}\mathbf{1}_{C_a}=\sum_{c\in[k]}(W^{u})_{ca}\,\mathbf{1}_{C_c}\,.
\]
\end{lemma}

\begin{proof}
We argue by induction on $u\geq1$. To establish the result for $u=1$, we argue as follows. Since $\mathbb{E}[A^{\top}]=\frac{1}{n}\sum_{c,d\in[k]}\mathcal{W}_{cd}\,\mathbf{1}_{C_c}\mathbf{1}_{C_d}^{\top}$ by \eqref{eq:gen:EA:Wbar}, we have 
\[
\mathbb{E}[A^{\top}]\,\mathbf{1}_{C_a}
=\frac{1}{n}\sum_{c,d\in[k]}\mathcal{W}_{cd}\,\mathbf{1}_{C_c}\big(\mathbf{1}_{C_d}^{\top}\mathbf{1}_{C_a}\big)
=\frac{1}{n}\sum_{c\in[k]}\mathcal{W}_{ca}\,|C_a|\,\mathbf{1}_{C_c}
=\sum_{c\in[k]}\mathcal{W}_{ca}\,\alpha_a\,\mathbf{1}_{C_c}\,,
\]
where in the second and third equalities we used  $\mathbf{1}_{C_d}^{\top}\mathbf{1}_{C_a}=|C_a|\,\mathbbm{1}_{\{d=a\}}$ and $|C_a|=n\alpha_a$, respectively. 
Using that $W_{ca}=\mathcal{W}_{ca}\alpha_a$ by\eqref{eq:gen:W:Wbar}, we see that the right-hand side is equal to
$\sum_{c\in[k]}W_{ca}\mathbf{1}_{C_c}$, which is the claim for $u=1$.

Now, assume that the result holds for some $u\geq1$. Using linearity and the
inductive hypothesis,
\[
\big(\mathbb{E}[A^{\top}]\big)^{u+1}\mathbf{1}_{C_a}
=\mathbb{E}[A^{\top}]\Big(\sum_{c\in[k]}(W^{u})_{ca}\,\mathbf{1}_{C_c}\Big)
=\sum_{c\in[k]}(W^{u})_{ca}\,\mathbb{E}[A^{\top}]\,\mathbf{1}_{C_c}\,.
\]
Finally, applying the base case to each $\mathbb{E}[A^{\top}]\mathbf{1}_{C_c}$ gives
\[
\big(\mathbb{E}[A^{\top}]\big)^{u+1}\mathbf{1}_{C_a}
=\sum_{c,d\in[k]}(W^{u})_{ca}\,W_{dc}\,\mathbf{1}_{C_d}
=\sum_{d\in[k]}\Big(\sum_{c\in[k]}W_{dc}(W^{u})_{ca}\Big)\mathbf{1}_{C_d}
=\sum_{d\in[k]}(W^{u+1})_{da}\,\mathbf{1}_{C_d}\,,
\]
which concludes the proof.
\end{proof}

Our last result focus on the mean of the matrix $L$.

\begin{lemma}\label{lem:gen:L}
Suppose Condition \eqref{eq:wellposed_random_scenario} is granted and let $\xi$ be the vector defined in  \eqref{def:xi:vector}.
For all $a,b\in [k]$, let $v^{(a,b)}$ be the vector defined in \eqref{def:gen:Vab}.
Let $W$ and $L$ be the matrices defined in \eqref{def:matrix_G} and  \eqref{def:gen:L}, respectively. 
Then the following identities hold.
\begin{itemize}
\item[(i)] For all $b\in[k]$,
\[
\mathbb{E}^{(n)}[L]\,{\rm vec}\big({\rm diag}(\mathbf{1}_{C_b})\big)
=\frac{1}{n}\sum_{a\in[k]}\sum_{c\in[k]}W_{ab}W_{bc}\,\alpha_c^{-1}\,v^{(a,c)}\,.
\]
\item[(ii)] For all $a,b\in[k]$,
\[
\mathbb{E}^{(n)}[L]\,v^{(a,b)}
=\sum_{c\in[k]}\sum_{d\in[k]}W_{ca}W_{db}\Big(1-\frac{1}{\alpha_a n}\mathbbm{1}_{\{a=b\}}\Big)v^{(c,d)}\,.
\]
\item[(iii)] For all $u\geq 1$,
\[
\big(\mathbb{E}^{(n)}[L]\big)^{u}\Big(\sum_{b\in[k]}\xi_b\,{\rm vec}\big({\rm diag}(\mathbf{1}_{C_b})\big)\Big)
=\frac{1}{n}\sum_{a\in[k]}\sum_{c\in[k]}\big(H^{(u)}\big)_{ac}v^{(a,c)}\,,
\]
where $H^{(1)}=W{\rm diag}(\xi)\,W\,{\rm diag}(\alpha^{-1})$ and,
for $u\geq 1$,
\begin{equation}\label{eq:gen:H:rec}
H^{(u+1)}=W\Big(H^{(u)}-\frac{1}{n}{\rm Diag}\big(H^{(u)}\big)\,{\rm diag}(\alpha^{-1})\Big)W^{\top}\,.
\end{equation}
\end{itemize}
\end{lemma}

\begin{proof}
We start by proving (i). For all $i,j\in[n]$, using the
definition of $L$, the independence of $(A_{\ell i})_{\ell\in [n]}$ and $(A_{\ell j})_{\ell\in [n]}$
for $i\neq j$, and the first equality of \eqref{eq:gen:EA:Wbar}, we have
\begin{align*}
\big(\mathbb{E}^{(n)}[L]\,{\rm vec}({\rm diag}(\mathbf{1}_{C_b}))\big)_{ij}
&=\mathbbm{1}_{\{i\neq j\}}\sum_{\ell\in[n]}\mathbb{E}^{(n)}[A_{\ell i}A_{\ell j}]\,\mathbbm{1}_{\{\ell\in C_b\}}
\\
&=\mathbbm{1}_{\{i\neq j\}}\sum_{\ell\in[n]}\mathbb{E}^{(n)}[A_{i\ell}]\mathbb{E}^{(n)}[A_{\ell j}]\,\mathbbm{1}_{\{\ell\in C_b\}}\,.
\end{align*}
By the second equality of \eqref{eq:gen:EA:Wbar}, we have $\mathbb{E}^{(n)}[A_{i\ell}]=\frac{1}{n}\mathcal{W}_{ab}$ for $i\in C_a,\ \ell\in C_b$,
and $\mathbb{E}^{(n)}[A_{\ell j}]=\frac{1}{n}\mathcal{W}_{bc}$ for $\ell\in C_b,\ j\in C_c$.
Hence, using that $|C_b|=n\alpha_b$, we obtain that
\begin{align*}
\sum_{\ell\in[n]}\mathbb{E}^{(n)}[A_{i\ell}]\mathbb{E}^{(n)}[A_{\ell j}]\,\mathbbm{1}_{\{\ell\in C_b\}}
&=\sum_{a\in[k]}\sum_{c\in[k]}\mathbbm{1}_{\{i\in C_a\}}\mathbbm{1}_{\{j\in C_c\}}
\,\frac{1}{n^2}\sum_{\ell\in C_b}\mathcal{W}_{ab}\mathcal{W}_{bc}\\
&=\frac{1}{n}\sum_{a\in[k]}\sum_{c\in[k]}\mathcal{W}_{ab}\mathcal{W}_{bc}\,\alpha_b
\big(\mathbf{1}_{C_a}\mathbf{1}_{C_c}^{\top}\big)_{ij}\\
&=\frac{1}{n}\sum_{a\in[k]}\sum_{c\in[k]}W_{ab}W_{bc}\,\alpha_c^{-1}\,\Big(\mathbf{1}_{C_a}\mathbf{1}_{C_c}^{\top}\Big)_{ij}\,,
\end{align*}
where in the last equality we used $\mathcal{W}_{ab}\alpha_b=W_{ab}$ and
$\mathcal{W}_{bc}=W_{bc}\alpha_c^{-1}$, which follow from~\eqref{eq:gen:W:Wbar}. 
Combining the above identities, we deduce that 
\[
\big(\mathbb{E}^{(n)}[L]\,{\rm vec}({\rm diag}(\mathbf{1}_{C_b}))\big)_{ij}=\frac{1}{n}\sum_{a\in[k]}\sum_{c\in[k]}W_{ab}W_{bc}\,\alpha_c^{-1}\,\mathbbm{1}_{\{i\neq j\}}\Big(\mathbf{1}_{C_a}\mathbf{1}_{C_c}^{\top}\Big)_{ij},
\]
proving (i).

We now prove (ii). 
We only consider the case $a=b$. The case $a\neq b$ is treated similarly. Hence, suppose $a=b$. Take $i,j\in[n]$ and note that, by the definition of $v^{(a,a)}$ and the independence of $(A_{\ell i})_{\ell\in [n]}$ and $(A_{h i})_{h\in [n]}$
for $i\neq j$,
\begin{align*}
\big(\mathbb{E}^{(n)}[L]v^{(a,a)}\big)_{ij}
&=\mathbbm{1}_{\{i\neq j\}}\sum_{\substack{(h,\ell)\in C_a\times C_a}}\mathbbm{1}_{\{h\neq\ell\}}\mathbb{E}^{(n)}[A_{\ell j}A_{hi}]\\
&=\mathbbm{1}_{\{i\neq j\}}\sum_{\substack{(h,\ell)\in C_a\times C_a}}\mathbbm{1}_{\{h\neq\ell\}}\mathbb{E}^{(n)}[A_{hi}]\mathbb{E}^{(n)}[A_{\ell j}]\\
&=\mathbbm{1}_{\{i\neq j\}}\Big[\Big(\sum_{h\in C_a}\mathbb{E}^{(n)}[A_{hi}]\Big)\Big(\sum_{\ell\in C_a}\mathbb{E}^{(n)}[A_{\ell j}]\Big)
-\sum_{h\in C_a}\mathbb{E}^{(n)}[A_{hi}]\mathbb{E}^{(n)}[A_{hj}]\Big]\,.
\end{align*}
Next, using \eqref{eq:gen:EA:Wbar}, we can check that for $i\in C_b$ and $j\in C_c$,
\[
\sum_{h\in C_a}\mathbb{E}^{(n)}[A_{hi}]=\alpha_a\mathcal{W}_{ab}\, \ \text{and} \ 
\sum_{h\in C_a}\mathbb{E}^{(n)}[A_{hi}]\mathbb{E}^{(n)}[A_{hj}]=\frac{\alpha_a}{n}\mathcal{W}_{ab}\mathcal{W}_{ac}.
\]
Putting together the above identities, we then obtain that 
\[
\big(\mathbb{E}^{(n)}[L]v^{(a,a)}\big)_{ij}
=\sum_{b\in[k]}\sum_{c\in[k]}\Big(\alpha_a^2\mathcal{W}_{ab}\mathcal{W}_{ac}
-\frac{\alpha_a}{n}\mathcal{W}_{ab}\mathcal{W}_{ac}\Big)\mathbbm{1}_{\{i\neq j\}}\big(\mathbf{1}_{C_b}\mathbf{1}_{C_c}^{\top}\big)_{ij}\,.
\]
Using $W_{ba}=\mathcal{W}_{ab}\alpha_a$ and $W_{ca}=\mathcal{W}_{ac}\alpha_a$
(see~\eqref{eq:gen:W:Wbar} and the symmetry of $\mathcal{W}$), the term inside the bracket is equal to
$W_{ba}W_{ca}\big(1-\tfrac{1}{n\alpha_a}\big)$, whence
\[
\mathbb{E}^{(n)}[L]v^{(a,a)}=\sum_{b\in[k]}\sum_{c\in[k]}W_{ba}W_{ca}\Big(1-\frac{1}{n\alpha_a}\Big)v^{(b,c)}\,,
\]
showing (ii) for $a=b$.

Finally, we prove (iii) by induction on $u\geq 1$. For $u=1$, using (i), we have
\begin{align*}
\mathbb{E}^{(n)}[L]\Big(\sum_{b\in[k]}\xi_b\,{\rm vec}({\rm diag}(\mathbf{1}_{C_b}))\Big)
&=\sum_{b\in[k]}\xi_b\,\frac{1}{n}\sum_{a\in[k]}\sum_{c\in[k]}W_{ab}W_{bc}\,\alpha_c^{-1}\,v^{(a,c)}\\
&=\frac{1}{n}\sum_{a\in[k]}\sum_{c\in[k]}\Big(\sum_{b\in[k]}W_{ab}\xi_b W_{bc}\,\alpha_c^{-1}\Big)v^{(a,c)}\\
&=\frac{1}{n}\sum_{a\in[k]}\sum_{c\in[k]}\big(H^{(1)}\big)_{ac}v^{(a,c)}\,,
\end{align*}
with $H^{(1)}=W{\rm diag}(\boldsymbol{\xi})W{\rm diag}(\boldsymbol{\alpha^{-1}})$. Next, suppose the
identity holds for some $u\geq 1$. Using the inductive hypothesis and (ii), we obtain that
\begin{align*}
\big(\mathbb{E}^{(n)}[L]\big)^{u+1}\Big(\sum_{b\in [k]}\xi_b\,{\rm vec}({\rm diag}(\mathbf{1}_{C_b}))\Big)
&=\frac{1}{n}\sum_{a,c\in [k]}\big(H^{(u)}\big)_{ac}\,\mathbb{E}^{(n)}[L]v^{(a,c)}\\
&=\frac{1}{n}\sum_{c',d\in [k]}\Big(\sum_{a,c\in [k]}W_{c'a}\big(H^{(u)}\big)_{ac}W_{dc}\Big(1-\tfrac{1}{\alpha_a n}\mathbbm{1}_{\{a=c\}}\Big)\Big)v^{(c',d)}.
\end{align*}
Observing that 
\begin{align*}
\sum_{a,c\in [k]}W_{c'a}\big(H^{(u)}\big)_{ac}W_{dc}\Big(1-\tfrac{1}{\alpha_a n}\mathbbm{1}_{\{a=c\}}\Big)
&=\Big(WH^{(u)}W^{\top}
-\frac{1}{n}W{\rm Diag}(H^{(u)})\,{\rm diag}(\alpha^{-1})W^{\top}\Big)_{c'd}\\
&=H^{(u+1)}_{c'd},
\end{align*}
we then conclude that 
\[
\big(\mathbb{E}^{(n)}[L]\big)^{u+1}\Big(\sum_{b\in [k]}\xi_b\,{\rm vec}({\rm diag}(\mathbf{1}_{C_b}))\Big)=\frac{1}{n}\sum_{c',d\in [k]}H^{(u+1)}_{c'd}v^{(c',d)},
\]
proving (iii).
\end{proof}

\section{Proof of Lemma \ref{lem:random_NFC_welldefined}}
\label{sec:tech:lemmas}

In this section, we prove Lemma \ref{lem:random_NFC_welldefined}.



\begin{proof}[Proof of Lemma~\ref{lem:random_NFC_welldefined}]
Let us begin observing that 
\[
\mathbb{P}^{(n)}((\Theta,H)\in \Omega^{adm}) = \mathbb{P}^{(n)}\big((\Theta,H): \lambda \in \big(\vertiii{A^{\top}_{(\Theta,H)}}_{\infty}, 1-\vertiii{A^{\top}_{(\Theta,H)}}_{\infty}\big) \big)\,.
\]
Since $H_{ij} \in \{-1, +1\}$, in order to show that $\lambda \in \big(\vertiii{A^{\top}_{(\Theta,H)}}_{\infty}, 1-\vertiii{A^{\top}_{(\Theta,H)}}_{\infty}\big)  $
it suffices that  
\begin{equation}\label{eq:suff}
\max_{a \in [k]}\max_{ i \in C_a}\Big\{\frac{\mu_a}{n} \sum_{j:j \overset{a}{\sim} i} \Theta_{ji}  + \frac{\Out}{n} \sum_{j: j \nsim i} \Theta_{ji}\Big\} <\lambda < 1- \max_{a \in [k]}\max_{ i \in C_a}\Big\{\frac{\mu_a}{n} \sum_{j:j \overset{a}{\sim} i} \Theta_{ji}  + \frac{\Out}{n} \sum_{j: j \nsim i} \Theta_{ji}\Big\} \,.
\end{equation}

Using that $\mathbb{E}^{(n)}[\Theta_{ji}]=p_a$ if $i \overset{a}{\sim} j$ for $a \in [k]$ and $\mathbb{E}^{(n)}[\Theta_{ji}]=q$ if $i \nsim j$, and defining for $a \in [k]$ and $i \in C_a$
\[
Y_{a,i}:=\frac{\mu_a}{n} \sum_{j:j \overset{a}{\sim} i} \big(\Theta_{ji}-p_a\big) + \frac{\Out}{n} \sum_{j: j \nsim i} \big(\Theta_{ji}-q\big)\,,
\]
we can rewrite \eqref{eq:suff}  (recall that $|C_a|=\alpha_a n$)  as
\[
\max_{a \in [k]}\big\{\max_{ i \in C_a} Y_{a,i} + \mu_a\alpha_a p_a + \Out(1-\alpha_a) q  \big\} <\lambda < 1- \max_{a \in [k]}\big\{\max_{ i \in C_a} Y_{a,i} + \mu_a\alpha_a p_a + \Out(1-\alpha_k) q  \big\} \,.
\]
Setting  $M=\max_{a \in [k]}\{\mu_a\alpha_a p_a+\Out(1-\alpha_a)q\}$, we obtain that  to verify the claim,  it is enough to show 
\begin{equation}\label{eq:suff2}
  \mathbb{P}^{(n)}\Big( M + \max_{a \in [k]}\max_{ i \in C_a} Y_{a,i}  <\lambda < 1- M - \max_{a \in [k]}\max_{ i \in C_a} Y_{a,i}\Big)\geq 1-2n e^{-cn}\,,  
\end{equation}
where is a constant $c>0$ independent of $n$, depending on the parameters $\lambda, (\alpha_a)_{a \in [k]}$, $(\mu_a)_{a \in [k]}$, $\Out, (p_a)_{a \in [k]}$ and $q$.

Observe that, for any fixed $a \in [k]$ and any $i \in C_a$,  the random variables $\{\frac{\mu_a}{n}(\Theta_{ji}-p_a)\}_{j\overset{a}{\sim} i}$ are i.i.d. and assume values in the bounded interval  $[-\frac{\mu_a}{n}p_a, \frac{\mu_a}{n}(1-p_a) ]$. Thus, by Hoeffding's Lemma, they are subgaussian with variance factor $\frac{\mu_a^2}{4n^2}$.  Moreover, by independence, we have that $\frac{\mu_a}{n} \sum_{j: j \overset{a}{\sim} i} \big(\Theta_{ji}-p_a\big)$ is subgaussian with variance factor $\frac{\mu_a^2}{4n^2}|C_a|$. 
Similarly, we have that  $\frac{\Out}{n} \sum_{j \nsim i} \big(\Theta_{ji}-q\big)$ is subgaussian with variance factor $\frac{\Out^2}{4n^2}(n-|C_a|)$. Using again independence, we obtain that for any $a \in [k]$ and $i \in C_a$ the random variable $Y_{a,i}$ is subgaussian with parameter  $\frac{\mu_a^2}{4n}\alpha_a + \frac{\Out^2}{4n}(1-\alpha_a)$. Setting  
$\sigma^2:=\max_{a \in [k]}\{\mu_a^2\alpha_a + \Out^2 (1-\alpha_a)  \}$, we obtain that the random variables $\{Y_{a,i}\}_{a \in [k], i \in C_a}$ are subgaussians with variance factor $\frac{1}{4n}\sigma^2$. Moreover, $\mathbb{E}^{(n)}[Y_{a,i}]=0$, for all $a \in [k]$ and all $i \in C_a$.  Thus, for all for all $a \in [k], i \in C_a$ and all $y>0$, it holds that 
\[
\mathbb{P}^{(n)}(Y_{a,i}>y) \leq  e^{-\frac{y^2 4n}{2 \sigma^2}} \quad \overset{\text{Union bound}}{\implies} \quad \mathbb{P}^{(n)}\big(\max_{a \in [k]}\max_{i \in C_a}Y_{a,i}>y\big) \leq n e^{-\frac{y^2 4n}{2 \sigma^2}}\,.
\]
By the assumption  $M<\lambda<1-M$, setting  first $y=\lambda -M>0$ and then $y = 1-M-\lambda>0$,   we obtain 
\begin{align*}
\mathbb{P}^{(n)}(M + \max_{a \in [k]}\max_{i \in C_a}Y_{a,i}>\lambda ) & \leq ne^{-\frac{(\lambda - M)^2 4}{2 \sigma^2}n}\,,
\\
\mathbb{P}^{(n)}(\lambda >1-M - \max_{a \in [k]}\max_{i \in C_a}Y_{a,i} ) &\leq ne^{-\frac{(1-M-\lambda)^2 4}{2 \sigma^2}n}\,,
\end{align*}
which combined imply \eqref{eq:suff2}. 
\end{proof}

\section{Community identifiability}
\label{app:proof:sec:community:ident}

In this section, we provide the proofs of the results of Section \ref{sec:detection}.

We start proving Lemma \ref{lemma:identificability:through:the:mean}. 

\begin{proof}[Proof of Lemma \ref{lemma:identificability:through:the:mean}]
Suppose $W\mathbf{1}_k=c\mathbf{1}_k$. Since \eqref{eq:wellposed_random_scenario} is granted, then necessarily $|c| <1$.  In this case, we have $W^u\mathbf{1}_k=c^u\mathbf{1}_k$ for all $u\geq 0$, so that $\zeta=(\sum_{u=0}^{\infty}c^u)\mathbf{1}_k=(1-c)^{-1}\mathbf{1}_k$. 

For the other direction, suppose $\zeta=c\mathbf{1}_k$. Since \eqref{eq:wellposed_random_scenario} is granted, $\zeta$ solves the equation $\zeta=\mathbf{1}_k+W\zeta$.  It then follows that $c\mathbf{1}_k=\mathbf{1}_k+cW\mathbf{1}_k$, and  \eqref{eq:wellposed_random_scenario} assures that $c>0$. Then, we must have that  $W\mathbf{1}_k=\frac{c-1}{c}\mathbf{1}_k$.
\end{proof}

Next, we prove Proposition \ref{prop:limit:lagged:cov:kappa}. In the proof, we shall use the following proposition, which is proved at the end  of this section.

\begin{proposition}
\label{prop:properties:limit:cov:matrix}
Let $W$ be the matrix defined in \eqref{def:matrix_G}. 
Suppose \eqref{eq:wellposed_random_scenario} and \eqref{def:kappa} are granted. Denote $\xi=\lambda\zeta_1(1-\lambda\zeta_1)$. Let $\gamma$ and $\rho_{\nsim}$ as defined in \eqref{def:gamma:rhoout}. Then, the matrix $\Psi$ defined in \eqref{def:Psi} satisfies
\begin{equation}\label{eq:Psi:kappa:closed}
\Psi=\xi_1 {\cal G}_{\geq 1}(W^2){\rm diag}(\alpha^{-1})=\xi_1\big(\gamma \, {\rm diag}(\alpha^{-1})+\rho_{\nsim}\mathbf{1}_k(\mathbf{1}_k)^{\top}\big)\,,   
\end{equation}
where $\alpha^{-1}=(\alpha^{-1}_1,\ldots,\alpha^{-1}_k)$ and $\xi_1$ is defined in \eqref{def:xi:vector}. 
\end{proposition}

\begin{proof}[Proof of Proposition \ref{prop:limit:lagged:cov:kappa}]

First, summing $W_{ab}=(2\beta-1)\big(\mu_b\alpha_b p_b\mathbbm{1}_{\{a=b\}}+\Out\alpha_b q\mathbbm{1}_{\{a\neq b\}}\big)$
over $b\in [k]$ and using that 
$\mu_a\alpha_a p_a+\Out q(1-\alpha_a)=\kappa+\Out q=M$ for all $a\in [k]$ (this holds by \eqref{def:kappa}),
we deduce that $(W\mathbf{1}_k)_a=(2\beta-1)M$ for all $a\in [k]$, i.e. $W\mathbf{1}_k=(2\beta-1)M\,\mathbf{1}_k.$ As a consequence of this identity, we obtain that for all $t\geq 0$,
\begin{equation}\label{eq:W:one:kappa}
W^{t}\mathbf{1}_k\mathbf{1}_k^{\top}=(2\beta-1)^{t}M^{t}\,\mathbf{1}_k\mathbf{1}_k^{\top}.
\end{equation}
Second, observe that by \eqref{eq:gen:commute}, we have that $W{\rm diag}(\alpha^{-1})=\mathcal{W}$. Using this,
\eqref{def:gen:Wbar} and the fact that $\Out q=M-\kappa$, together with $\mu_a p_a-\Out q=\kappa\alpha_a^{-1}$ (which holds by \eqref{def:kappa}), we then deduce that
\begin{equation}\label{eq:cW:kappa}
W{\rm diag}(\alpha^{-1})=\mathcal{W}=(2\beta-1)\kappa\,{\rm diag}(\alpha^{-1})+(2\beta-1)(M-\kappa)\,\mathbf{1}_k\mathbf{1}_k^{\top}\,,
\end{equation}

{\bf Claim.} For every integer $t\geq0$,
\begin{equation}\label{eq:Wt:diag:kappa}
W^{t}{\rm diag}(\alpha^{-1})
=(2\beta-1)^{t}\kappa^{t}\,{\rm diag}(\alpha^{-1})
+(2\beta-1)^{t}\big(M^{t}-\kappa^{t}\big)\,\mathbf{1}_k\mathbf{1}_k^{\top}\,.
\end{equation}

{\it Proof of the Claim.}
We argue by induction on $t$. The case $t=0$ is immediate. Moreover, note that the case $t=1$ is \eqref{eq:cW:kappa}. Now, assume \eqref{eq:Wt:diag:kappa}
holds for some $t\geq1$. Using that $W{\rm diag}(\alpha^{-1})=\mathcal W$,
\eqref{eq:cW:kappa} and \eqref{eq:W:one:kappa} with $t=0$, we obtain
\begin{align*}
W^{t+1}{\rm diag}(\alpha^{-1})
&=W\big(W^{t}{\rm diag}(\alpha^{-1})\big)\\
&=W\big((2\beta-1)^{t}\kappa^{t}{\rm diag}(\alpha^{-1})+(2\beta-1)^{t}(M^t-\kappa^t)\mathbf{1}_k\mathbf{1}_k^{\top}\big)\\
&=(2\beta-1)^{t}\kappa^{t}\,\mathcal{W}
+(2\beta-1)^{t}\big(M^{t}-\kappa^{t}\big)(2\beta-1)M\,\mathbf{1}_k\mathbf{1}_k^{\top}\\
&=(2\beta-1)^{t+1}\kappa^{t+1}\,{\rm diag}(\alpha^{-1})
+(2\beta-1)^{t+1}\Big[\kappa^{t}(M-\kappa)+M\big(M^{t}-\kappa^{t}\big)\Big]\mathbf{1}_k\mathbf{1}_k^{\top}\,.
\end{align*}
Combining the above identity with the fact that $\kappa^{t}(M-\kappa)+M(M^{t}-\kappa^{t})=M^{t+1}-\kappa^{t+1}$, the claim for $t+1$ follows.
\qed

Next, combining \eqref{eq:Psi:kappa:closed}, \eqref{eq:Wt:diag:kappa} and \eqref{eq:W:one:kappa}, we obtain that for all integer $t\geq 0$,
\begin{align*}
W^{t}\Psi
&=\xi_1\gamma\,W^{t}{\rm diag}(\alpha^{-1})+\xi_1\rho_{\nsim}\,W^{t}\mathbf{1}_k\mathbf{1}_k^{\top}\\
&=\xi_1(2\beta-1)^{t}\kappa^{t}\gamma\,{\rm diag}(\alpha^{-1})
+\xi_1(2\beta-1)^{t}\big(\gamma(M^{t}-\kappa^{t})+\rho_{\nsim}M^{t}\big)\mathbf{1}_k\mathbf{1}_k^{\top}\,,
\end{align*}
that is, for all $a,b\in[k]$,
\begin{equation}\label{eq:WtPsi:entry}
(W^{t}\Psi)_{ab}
=\xi_1(2\beta-1)^{t}\kappa^{t}\gamma\,\alpha_a^{-1}\mathbbm{1}_{\{a=b\}}
+\xi_1(2\beta-1)^{t}\big(\gamma(M^{t}-\kappa^{t})+\rho_{\nsim}M^{t}\big)\,.
\end{equation}

Substituting~\eqref{eq:WtPsi:entry} into
$\frac1n\sum_{a,b\in[k]}(W^{t}\Psi)_{ab}\mathbf{1}_{C_a}\mathbf{1}_{C_b}^{\top}$ and using
$\sum_{a,b}\mathbf{1}_{C_a}\mathbf{1}_{C_b}^{\top}=\mathbf{1}_n\mathbf{1}_n^{\top}$, one can then check that for all integer $t\geq 0$, 
\begin{equation}\label{eq:block:kappa}
\frac1n\sum_{a,b\in[k]}(W^{t}\Psi)_{ab}\,\mathbf{1}_{C_a}\mathbf{1}_{C_b}^{\top}
=\xi_1\big(S^{(t)}+\overline{B}^{(t)}\big)\,,\quad t\geq 0.
\end{equation}
Since $\overline{\Sigma}^{(0)}$ corresponds to the left-hand side
of \eqref{eq:block:kappa} with $t=0$ (note that $W^{0}=I_k$), we conclude that $\overline{\Sigma}^{(0)}=\xi_1(S^{(0)}+\overline{B}^{(0)})$.

Finally, note that under~\eqref{def:kappa} the entries of $\zeta$ are all equal
(see Lemma~\ref{lemma:identificability:through:the:mean}), so that $\xi=\xi_1\mathbf{1}_k$
and $\overline{v}=\xi_1\mathbf{1}_n$, which in turn implies ${\rm diag}(\overline{v})=\xi_1 I_n$. Using this and \eqref{eq:block:kappa} for $t\geq 1$, and recalling the definition of $\overline{\Sigma}^{(t)}_{(\Theta,H)}$ in \eqref{def:gen:Gamma:t}, we obtain that  
\begin{align*}
\overline{\Sigma}^{(t)}_{(\Theta,H)}&=(A^{\top})^{t}{\rm diag}(\overline{v})
+\frac1n\sum_{a,b}(W^{t}\Psi)_{ab}\mathbf{1}_{C_a}\mathbf{1}_{C_b}^{\top}=\xi_1\big((A^{\top})^{t}+S^{(t)}+\overline{B}^{(t)}\big)\,,
\end{align*}
and the result follows.
\end{proof}

Finally, we prove Proposition ~\ref{prop:properties:limit:cov:matrix}.

\begin{proof}[Proof of Proposition~\ref{prop:properties:limit:cov:matrix}]

 We will begin showing that $\Psi=\xi_1 {\cal G}_{\geq 1}(W^2){\rm diag}(\alpha^{-1})$. To show this identity, first note that under Condition~\eqref{def:kappa}, we have $\xi=\xi_1\mathbf{1}_k$. Then, use this and \eqref{eq:gen:commute} to obtain that
\begin{equation*}
\Psi=\xi_1\sum_{u=1}^{\infty}W^{u}{\rm diag}(\alpha^{-1})(W^{\top})^{u}=\xi_1\sum_{u=1}^{\infty}W^{2u}{\rm diag}(\alpha^{-1})
=\xi_1\,{\cal G}_{\geq 1}(W^{2})\,{\rm diag}(\alpha^{-1})\,,
\end{equation*}
concluding the first half of the proof.

Hence, to conclude the proof, it suffices to show that ${\cal G}_{\geq 1}(W^{2})\,{\rm diag}(\alpha^{-1})=\gamma \, {\rm diag}(\alpha^{-1})+\rho_{\nsim}\mathbf{1}_k(\mathbf{1}_k)^{\top}$.
This is can be done as follows. 

First of all, let us check that the matrix $W^2\text{diag}(\alpha^{-1})$ is invertible. To that end, first observe that by \eqref{eq:Wt:diag:kappa} with $t=2$, we can rewrite $W^2\text{diag}(\alpha^{-1})$ as 
\begin{equation*}
W^2\text{diag}(\alpha^{-1})=(2\beta-1)^2\kappa^2\text{diag}(\alpha^{-1})+\phi_{\nsim}\mathbf{1}_k(\mathbf{1}_{k})^{\top}\,,    
\end{equation*}
where $\phi_{\nsim}=(2\beta-1)^2(\mu_{\nsim}q)(\mu_{\nsim}q+2\kappa)$.
Hence, the matrix $W^2\text{diag}(\alpha^{-1})$ can be written as a sum of an invertible matrix and a rank 1 matrix. Since
\begin{equation*}
1+\phi_{\nsim}\mathbf{1}^\top_k  \big((2\beta-1)^2\kappa^2)^{-1}\text{diag}(\alpha)\big)\mathbf{1}_k=\frac{(\mu_{\nsim}q+\kappa)^2}{\kappa^2}=\frac{(W\mathbf{1}_k)_1}{(2\beta-1)^2 \kappa^2}\neq 0\,,
\end{equation*}
it then follows from Sherman-Morrison formula that $W^2\text{diag}(\alpha^{-1})$ is invertible and its inverse is given by  
\begin{equation}
\label{def:inverse:W2diag}
\frac{1}{(2\beta-1)^2\kappa^2}\Big(\text{diag}(\alpha)-\frac{\mu_{\nsim}q(\mu_{\nsim}q+2\kappa)}{(\mu_{\nsim}q+\kappa)^2}\alpha(\alpha)^{\top}\Big)\,,   
\end{equation}
where $\alpha=(\alpha_1,\ldots,\alpha_k)$.
Since $\text{diag}(\alpha)$ is clearly invertible, the above argument also proves that $W^2$ is invertible.

Next, we use \eqref{def:inverse:W2diag} and the fact that $W^2$ is invertible to conclude the proof of the proposition.
Using that 
$
{\cal G}_{\geq 1}(W^2)={\cal G}(W^2)W^2    
$
together with \eqref{def:inverse:W2diag} and the fact that $W^2$ and $W^2\text{diag}(\alpha^{-1})$ are invertible, we deduce that ${\cal G}_{\geq 1}(W^2)\text{diag}(\alpha^{-1})$ is also invertible and its inverse is 
\begin{align*}
(W^2\text{diag}(\alpha^{-1}))^{-1}({\cal G}(W^2))^{-1}&=(W^2\text{diag}(\alpha^{-1}))^{-1}(I_k-W^2)
\\
&=(W^2\text{diag}(\alpha^{-1}))^{-1}I_k-\text{diag}(\alpha)(W^2)^{-1}W^2\\
&=(W^2\text{diag}(\alpha^{-1}))^{-1}-\text{diag}(\alpha)\\
&=\frac{1-(2\beta-1)^2\kappa^2}{(2\beta-1)^2\kappa^2}\text{diag}(\alpha)-\frac{\mu_{\nsim}q(\mu_{\nsim}q+2\kappa)}{(\mu_{\nsim}q+\kappa)^2(2\beta-1)^2\kappa^2}\alpha(\alpha)^{\top}\,. 
\end{align*}
 
Using the  identity $M=\kappa+\mu_{\nsim}q$ (thus,  $\kappa<M$) and the fact that when \eqref{eq:wellposed_random_scenario} holds $0<M< 1/2$, we have that 
\[
1 -  \frac{\mu_{\nsim}q(\mu_{\nsim}q+2\kappa)}{(\mu_{\nsim}q+\kappa)^2(1-(2\beta-1)^2\kappa^2)}=\frac{k^2(1- M^2(2\beta -1)^2)}{M^2(1- k^2(2\beta -1)^2)}\neq 0\,.
\]
Thus, applying  Sherman-Morrison formula in the above matrix and setting $\gamma=\frac{(2\beta -1)^2 \kappa^2}{1-(2\beta-1)^2 \kappa^2}$, we obtain that 
\[
{\cal G}_{\geq 1}(W^2)\text{diag}(\alpha^{-1}) = \gamma \text{diag}(\alpha^{-1}) + \rho_{\nsim} \mathbf{1}_k\mathbf{1}_k^\top \,,
\]
where $\rho_{\nsim}=\gamma((M/\kappa)^2-1)(1-(2\beta-1)^2M^2)^{-1}$, concluding the proof.
\end{proof}

\section{Proof of the preliminary results}
\label{sec:appendix:proof:priliminary:results}

In the section, we prove the results stated in Section \ref{sec:preliminary:results}, namely, 
Propositions \ref{prop:mean}, \ref{prop:cov:vec:form} and \ref{prop:t-laggedCov}. 

We start proving Proposition \ref{prop:mean}.

\begin{proof}[Proof of Proposition \ref{prop:mean}]
For each $i \in \left[ n \right]$, denote $m_i=\mathbb{E}_G\left[ X_{i,t}\right]$ and observe that $m_i$ can be writen as follows: 
\begin{equation*}
    m_{i} = \mathbb{P}_G\left( X_{i,t}=1 \right) = \mathbb{E}_G\left[ \mathbb{P}_G\left( X_{i,t}=1 | X_{t-1} \right) \right]\,.
\end{equation*}
Using \eqref{eq:firingProb} in the latter expression, we obtain that 
\begin{equation}\label{eq:average}
    \begin{split}
        m_{i} & = \mathbb{E}_G\Big[ \lambda + \sum_{a \in [k]} \mathbbm{1}_{\{i \in C_a\}} \Big(\frac{\mu_a}{n}\sum_{j: j\overset{a}{\sim}i}  \theta_{ji} \eta_{ji} \, X_{j,t-1} + \frac{\Out}{n} \sum_{j: j \nsim i} \theta_{ji} \eta_{ji} \,  X_{j,t-1} \Big)  \Big] 
        \\
        & = \lambda + \sum_{a \in [k]} \mathbbm{1}_{\{i \in C_a\}} \Big( \frac{\mu_a}{n}\sum_{j: j\overset{a}{\sim}i}  \theta_{ji} \eta_{ji} \mathbb{E}_G\left[ X_{j, t-1} \right] + \frac{\Out}{n} \sum_{j: j \nsim i} \theta_{ji} \eta_{ji}\mathbb{E}_G\left[ X_{j, t-1} \right] \Big)
        \\
        &=
        \lambda + \sum_{a \in [k]} \mathbbm{1}_{\{i \in C_k\}} \Big(\frac{\mu_a}{n}\sum_{j:j\overset{a}{\sim}i}  \theta_{ji} \eta_{ji}m_j + \frac{\Out}{n} \sum_{j: j \nsim i} \theta_{ji} \eta_{ji}m_j\Big)
        \,,
    \end{split}
\end{equation}

where in the third equality we used the assumption that the Markov chain $\{X_t\}_{t\geq 0}$ is stationary implying that $m_{i}=\mathbb{E}_G[X_{i,t}]$, for all $t$. 
Since \eqref{eq:average} holds for every $i \in [n]$, setting $m=( m_{1}, \cdots, m_{n})^{\top}=\Ex_G[X_0]$, we can rewrite the system of equation in a matrix form as follows:
\begin{equation*}
    m = \lambda \mathbf{1}_{n} + A^{\top} m  \iff  \big( I_n - A^{\top} \big)m  = \lambda \mathbf{1}_{n}\,.
\end{equation*}
Under the assumption that \eqref{eq:wellposed} holds, the matrix $\left(I_n - A^{\top} \right)$ is invertible  and  the claim follows.
\end{proof}

Next, we will prove Proposition \ref{prop:cov:vec:form}. As we will see, the proof that we present relies on the following result. 
%

\begin{proposition}[Synchronous firing covariance]\label{prop:cov}
Let $[n]$ be the set of neurons, $\mathcal{C}=\{C_1, \ldots,  C_k\}$  the $k$ communities, $G=(\mathcal{C}, \theta, \eta)$ the interaction graph and $\lambda\in[0,1]$ the spontaneous activity. Assume that the corresponding interaction  matrix $A$ and $\lambda$ satisfy~\eqref{eq:wellposed}. 
Then, the synchronous firing covariance matrix is the unique solution of the following matrix equation
\begin{equation}\label{eq:matrixEquation}
    \Sigma^{(0)}_G = d_0(A^{\top} \Sigma^{(0)}_G A) 
    + {\rm diag}(v_G(X_0))\,, 
\end{equation}
where $v_G(X_0)$ is the vector of the firing variances defined in \eqref{def:vector:variances}.
\end{proposition}

\begin{proof}
To avoid clutter in the notation, we shall omit the dependence on $G$ and simply write $\Sigma^{(0)}$. Let us begin observing that for $i = j$, $\Sigma^{(0)}_{ij}=Var_{G}(X_{i,0})$ and $d_0(A^{\top} \Sigma^{(0)} A))_{ij}=0$, thus \eqref{eq:matrixEquation} is satisfied for the diagonal elements. 
Let us denote $m_i=\Ex_G[X_{i,0}]$ for $i\in [n]$.
For $i \neq j$,  using the fact that the Markov chain $\{X_t\}_{t\geq 0}$ is stationary, we have that 
\begin{equation*}
    \begin{split}
     \Sigma^{(0)}_{ij}&=    Cov_{G} \left( X_{i,1}, X_{j,1} \right)  = \mathbb{E}_G \left[ X_{i,1} X_{j,1} \right] - \mathbb{E}_G \left[ X_{i,1} \right] \mathbb{E}_G \left[ X_{j,1} \right]  = \mathbb{E}_G\left[ \mathbb{E}_G \left[ X_{i,1}, X_{j,1} | X_{0} \right] \right] - m_{i}m_{j} \\
        & \overset{(\star)}{=} \mathbb{E}_G \big[ \mathbb{E}_G [ X_{i,1} | X_{0}]  \mathbb{E}_G[ X_{j,1} | X_{0}] \big] - m_{i}m_{j} \\
        & = \mathbb{E}_G \big[ \big( \lambda + ( A^{\top} X_{0})_{i} \big) \big( \lambda + ( A^{\top} X_{0})_{j} \big) \big] - \mathbb{E}_G \big[ \big( \lambda + ( A^{\top} X_{0})_{i} \big) \big] \mathbb{E}_G \big[ \big( \lambda + ( A^{\top} X_{0})_{j} \big) \big]\,, 
        \end{split}
        \end{equation*}
where in $(\star)$ we used the conditional independence of the model. Using further the linearity of the expectation, we obtain that 
\begin{equation*}
    \begin{split}
        \Sigma^{(0)}_{ij}&=\mathbb{E}_G\big[ ( A^{\top} X_{0})_{i} ( A^{\top} X_{0})_{j} \big] - \big( A^{\top} \mathbb{E}_G \big[ X_{0} \big] \big)_{i} \big( A^{\top} \mathbb{E}_G \big[ X_{0} \big] \big)_{j} \\
        & =\sum_{h\in [n]} \sum_{\ell\in [n]} A^{\top}_{ih} \mathbb{E} \left[ X_{h,0} X_{\ell,0} \right] A^{\top}_{j\ell} - \sum_{h\in [n]} A^{\top}_{ih} \mathbb{E}_G \left[ X_{h,0} \right] \sum_{\ell\in [n]} A^{\top}_{j\ell} \mathbb{E}_G\left[ X_{\ell,0} \right] \\
        & = \sum_{h\in [n]} \sum_{\ell\in [n]} A^{\top}_{ih} \left( \mathbb{E} \left[ X_{h,0} X_{\ell,0} \right] - \mathbb{E} \left[ X_{h,0} \right] \mathbb{E} \left[ X_{\ell,0} \right] \right) A^{\top}_{j\ell}\\  
        &= \sum_{h\in [n]} \sum_{\ell\in [n]} A^{\top}_{ih} \Sigma^{(0)}_{h\ell} A^{\top}_{j\ell} 
        = \big( A^{\top} \Sigma^{(0)} A \big)_{ij}=d_0\big( A^{\top} \Sigma^{(0)} A \big)_{ij}\,,
    \end{split}
\end{equation*}
and thus \eqref{eq:matrixEquation} is satisfied also for the  off-diagonal entries. 
\end{proof}

We now prove Proposition \ref{prop:cov:vec:form}.

\begin{proof}[Proof of Proposition \ref{prop:cov:vec:form}]
To avoid clutter in the notation, we shall omit the dependence on $G$ and simply write $\Sigma^{(0)}$.
We begin pointing out that the operator ${\rm vec}(\cdot)$ is linear and invertible. Therefore, the unique solution of \eqref{eq:matrixEquation} in a vectorized form is the solution of the following equation
\begin{equation*}
  {\rm vec}\big(\Sigma^{(0)}\big)= {\rm vec}\big(d_0(A^{\top} \Sigma^{(0)} A\big)) 
  +  {\rm vec}\big({\rm diag}(v_G(X_0))\big)\,.   
\end{equation*}

\noindent
{\bf Claim.} The matrix $L$ satisfies $L{\rm vec}(\Sigma^{(0)})={\rm vec}(d_0(A^{\top}\Sigma^{(0)} A))$. 

\medskip
\noindent
{\bf Proof of the claim.} Since $L_{uv}=0$ for all $v\in[n^2]$ and $u=n(i-1)+i$ with $i\in [n]$, we have that $L{\rm vec}(\Sigma^{(0)})_u=0=d_0(A^{\top}\Sigma^{(0)} A)_{ii}={\rm vec}(d_0(A^{\top}\Sigma^{(0)} A))_u$. Now, take $u=n(j-1)+i$ with $i,j\in[n]$ such that $i\neq j$ and observe that
\begin{align*}
L{\rm vec}(\Sigma^{(0)})_u &=\sum_{v\in[n^2]}L_{uv}{\rm vec}(\Sigma^{(0)})_v=\sum_{v\in[n^2]}(A^{\top}\otimes A^{\top})_{uv}{\rm vec}(\Sigma^{(0)})_v
\\
&=\sum_{h\in[n]}\sum_{\ell\in[n]} A^{\top}_{jh}A^{\top}_{i\ell}\Sigma^{(0)}_{\ell h}=\sum_{h\in [n]}\sum_{\ell\in [n]} A^{\top}_{i\ell}\Sigma^{(0)}_{\ell h}A_{hj}\\
&=(A^{\top}\Sigma^{(0)} A)_{ij}=d_0(A^{\top}\Sigma^{(0)} A)_{ij}={\rm vec}(d_0(A^{\top}\Sigma^{(0)} A))_u\,,
\end{align*}
concluding the proof of the claim.

Using the claim, the proof then simply follows from observing that $\big(I_{n^2}-L\big)$ is invertible. This can be seen by observing that  
\begin{align*}
\vertiii{L}_{\infty}=\max_{s} \sum_{u\in [n^2]} |L_{su}|=\max_{i \neq j} \sum_{h \in [n]}\sum_{\ell \in [n]} |A^\top_{ih}| |A^{\top}_{j\ell}|\leq \vertiii{A^\top}_{\infty}^2\,.  
\end{align*}
The above bound and the assumption that $A$ satisfies \eqref{eq:wellposed} implies that $\vertiii{L}_{\infty}<1$ which in turn assures that  $\big(I_{n^2}-L\big)$ is invertible. 
\end{proof}

Finally, we prove Proposition \ref{prop:t-laggedCov}.

\begin{proof}[Proof of Proposition \ref{prop:t-laggedCov}]
Let $m_i=\Ex_G[X_{i,0}]$ for $i\in[n]$. We first show that, for every $t\geq 1$,
\begin{equation}\label{eq:lag:recursion}
\Sigma^{(t)}_G=A^{\top}\Sigma^{(t-1)}_G\,.
\end{equation}

Fix $t\geq 1$ and $i,j\in[n]$. Using \eqref{eq:cond_ind} together with the Markov property of $\{X_t\}_{t\geq 0}$, we obtain
\begin{align*}
\mathbb{E}_G\big[X_{i,t}X_{j,0}\big]
&=\mathbb{E}_G\Big[X_{j,0}\,\mathbb{E}_G\big[X_{i,t}\mid X_0,\ldots,X_{t-1}\big]\Big]\\
&=\mathbb{E}_G\Big[X_{j,0}\,\mathbb{E}_G\big[X_{i,t}\mid X_{t-1}\big]\Big]=\mathbb{E}_G\big[X_{j,0}\big(\lambda+(A^{\top}X_{t-1})_i\big)\big]\,,
\end{align*}
where in the second equality we used $\mathbb{E}_G[X_{i,t}\mid X_{t-1}]=\phi_{G,i}(X_{t-1})=\lambda+(A^{\top}X_{t-1})_i$.
Expanding $(A^{\top}X_{t-1})_i=\sum_{\ell\in[n]}A^{\top}_{i\ell}X_{\ell,t-1}$ and using the
linearity of the expectation, it follows that
\begin{equation}
\label{eq1:proof:lagged:covariances}
\mathbb{E}_G\big[X_{i,t}X_{j,0}\big]=\lambda m_j+\sum_{\ell\in[n]}A^{\top}_{i\ell}\,\mathbb{E}_G\big[X_{\ell,t-1}X_{j,0}\big]\,.
\end{equation}
On the other hand, by stationarity $\mathbb{E}_G[X_{i,t}]=m_i$ for all $t\geq 0$, and conditioning on $X_{t-1}$ similarly
as above yields $m_i=\mathbb{E}_G[X_{i,t}]=\lambda+(A^{\top}m)_i=\lambda+\sum_{\ell\in[n]}A^{\top}_{i\ell}m_\ell$.
Hence, subtracting $m_i,m_j$ from both sides of \eqref{eq1:proof:lagged:covariances}, we deduce that 
\begin{align*}
\big(\Sigma^{(t)}_G\big)_{ij}
&=\mathbb{E}_G\big[X_{i,t}X_{j,0}\big]-m_im_j\\
&=\sum_{\ell\in[n]}A^{\top}_{i\ell}\Big(\mathbb{E}_G\big[X_{\ell,t-1}X_{j,0}\big]-m_\ell m_j\Big)=\sum_{\ell\in[n]}A^{\top}_{i\ell}\big(\Sigma^{(t-1)}_G\big)_{\ell j}=(A^{\top}\Sigma^{(t)}_G)_{ij}\,.
\end{align*}
Since $i,j\in[n]$ are
arbitrary, this proves \eqref{eq:lag:recursion}.
Finally, the result follows by iterating \eqref{eq:lag:recursion}.
\end{proof}

\section{A remark on Stein equations}

The goal of this section is to prove the result below. 
For each square matrix $B$, we write $\rho(B)$ to indicate its spectral radius.

\begin{lemma}\label{lem:Psi:well:defined}
Let $W$ be the square matrix defined in \eqref{def:matrix_G}. If $\rho(W)<1$, then the matrix $\Psi$ defined in \eqref{def:Psi} is well-defined and is the unique solution of Stein equation \eqref{def:stein:equation}.

\end{lemma}

\begin{proof}
Denote $C=W\,{\rm diag}\big(\xi\odot \alpha^{-1}\big)\,W^{\top}$ and observe that using the vectorization the Stein equation \eqref{def:stein:equation} can be rewritten as $(I_{k^2}-W\otimes W){\rm vec}(\Psi)={\rm vec}(C)$. 
Now, since $\rho(W)<1$ and the
eigenvalues of $W\otimes W$ are the products $\lambda_i\lambda_j$ of the eigenvalues of
$W$, we have $\rho(W\otimes W)=\rho(W)^2<1$ \cite[Theorem 4.2.12]{Horn1991:topics}. As a consequence, 
$(I_{k^2}-W\otimes W)$ is invertible
and its inverse is $\mathcal{G}(W\otimes W)=\sum_{u\ge 0}(W\otimes W)^{u}$.  
Moreover, since $(I_{k^2}-W\otimes W)$ is invertible, the equation $(I_{k^2}-W\otimes W){\rm vec}(\Psi)={\rm vec}(C)$ has a unique solution ${\rm vec}(\Psi)=\mathcal{G}(W\otimes W){\rm vec}(C)$. Finally, using that $(W\otimes W)^u=W^u\otimes W^u$ and $W^u\otimes W^u {\rm vec}(C)={\rm vec}(W^u C (W^{\top})^u)$ for all $u\geq 0$, we then conclude that
\[
{\rm vec}(\Psi)=\sum_{u\geq 0}{\rm vec}(W^u C (W^{\top})^u)={\rm vec}\Big(\sum_{u\geq 0}W^u C (W^{\top})^u\Big),
\]
so that we must have $\Psi=\sum_{u\geq 0}W^u C (W^{\top})^u=\sum_{u=1}^{\infty}W^{u}\,{\rm diag}\big(\xi\odot \alpha^{-1}\big)\,\big(W^{\top}\big)^{u}$, concluding the proof.
\end{proof}

\bibliographystyle{ACM-Reference-Format}
\bibliography{refs.bib}

\end{document}